\documentclass[reqno]{amsart}
\usepackage[T1]{fontenc}
\newtheorem{theorem}[subsection]{Theorem}

\usepackage[english]{babel}
\usepackage{parskip}
\usepackage[letterpaper,top=2cm,bottom=2cm,left=3cm,right=3cm,marginparwidth=1.75cm]{geometry}

\usepackage{amsmath}
\usepackage{graphicx}
\usepackage[colorlinks=true, allcolors=blue]{hyperref}
\usepackage{amsfonts}
\usepackage[utf8]{inputenc}
\usepackage{mathtools}
\usepackage{gensymb}
\usepackage{amsthm}
\usepackage{thmtools}
\usepackage{enumerate}
\usepackage{mfirstuc}
\usepackage{hyperref}
\declaretheorem[name=Theorem,numberwithin=section]{thm}
\declaretheorem[numberwithin=section]{lemma}
\declaretheorem[numberwithin=section]{corollary}
\declaretheorem[numberwithin=section,name=Definition]{defn}
\declaretheorem[name=Proposition,numberwithin=section]{prop}

\newcommand{\biggij}{\mbox{\normalfont\Large\bfseries \(g_{ij}\)}}

\newcommand{\rvline}{\hspace*{-\arraycolsep}\vline\hspace*{-\arraycolsep}}

\title{Symmetric Rearrangement and Geometric Inequalities on Riemannian Manifolds}

\begin{document}
\author{Richard Stone}
\date{}
\renewcommand*\contentsname{Summary}

\maketitle

\begin{abstract}
    This paper starts by introducing results from geometric measure theory to prove symmetric decreasing rearrangement inequalities on \(\mathbb{R}^n\), which give multiple proofs of the isoperimetric and P\'{o}lya-Szeg\H{o} inequalities. Then we consider smooth oriented Riemannian manifolds of the form \(M^n = (0,\infty)\times \Sigma^{n-1}\), and test what results carry over from the \(\mathbb{R}^n\) setting or what assumptions about \(M^n\) need to be added. Of particular interest was proving the smooth co-area formula in the Riemannian manifolds setting and re-formulating particular geometric inequalities.
\end{abstract}

\hypersetup{linkcolor=black}
\tableofcontents

\section{Introduction}\label{sec:intro}

Let \(A\subset \mathbb{R}^n\) be a measurable set of finite volume. \cite{short_course_RI}

\begin{defn}[Symmetric Rearrangement of a Set]
    The \textit{symmetric rearrangement} of \(A\) is the open centred ball, \(A^*\), whose volume agrees with \(A\), 
    \begin{equation}
        A^* = \{x\in \mathbb{R}^n : \omega _n|x|^n<\text{Vol}(A)\},
    \end{equation}
    where \(\omega _n = \frac{\pi^{n/2}}{\Gamma(\frac{n}{2}+1)}\) is the Lebesgue measure of the unit ball in \(\mathbb{R}^n\). We take \(\emptyset^* = \emptyset\) and \((\mathbb{R}^n)^* = \mathbb{R}^n\).
\end{defn}

Let \(f : \mathbb{R}^n \rightarrow \mathbb{R}\) be a non-negative measurable function which vanishes at infinity; that is all the super-level sets have finite measure. 

\begin{defn}[Symmetric Decreasing Rearrangement of a Function]
    The \textit{symmetric decreasing rearrangement} of \(f\) is the function \(f^* : \mathbb{R}^n \rightarrow \mathbb{R}\), defined as,
    \begin{equation}
        f^*(x) = \int_0^{\infty} \mathcal{X} _{\{y:f(y)>t\}^*}(x) dt.
    \end{equation}
\end{defn}
This is called the \textit{layer-cake decomposition} of \(f^*\) and its existence follows as a corollary of \textit{Cavalieri's Principle} in the section 3.1. Symmetrically rearranging a function is done by taking a level set of \(f\) and shifting this so it radially symmetric about the target axis and repeating for all the level sets. 


\(f^*\) is well-defined, since \(f^*\) is lower semicontinuous, as the level sets \(\{x:f^*(x)>t\}\) are all open for \(t>0\). Then we have that \(f^*\) is measurable, \cite{analysis}.

This symmetry idea generalises nicely onto smooth manifolds of the form \(M^n = (0,\infty)\times \Sigma^{n-1}\), where \(\Sigma ^{n-1}\) is an \((n-1)\)-dimensional smooth oriented manifold. Giving \((0,\infty )\) the natural orientation, \(M\) has the induced orientation.

Let \(\omega \in \Omega ^n (M)\) be an \(n\)-form on \(M\) such that for all \(p\in M\), 
\begin{equation}
    \omega (X^1, \dots , X^n) > 0 
\end{equation}
for all positively oriented bases \(X^1, \dots ,X^n\) of \(T_pM\). It may be useful to write,
\begin{equation}
    \omega (X^1, \dots ,X^n) = \varphi (x^1, \dots ,x^n) dx^1 \wedge \dots \wedge dx^n,
\end{equation}
where \(\varphi :M^n \rightarrow \mathbb{R}, \varphi >0\) is a smooth function on the manifold and \((x^i)\) are the corresponding coordinates for the \((X^i)\) basis.

\begin{defn}
     The \textit{symmetric rearrangement} of an open set \(A\subset M\) is the unique set \(A^* \subset M\) of the form \(A^* = (0,r_*)\times \Sigma\), such that
\end{defn}

\begin{equation}
    \int_ {A^*} \omega = \int_ A \omega < \infty,
\end{equation}
where we assume that,
\begin{equation}
    \lim_{r\rightarrow 0} \int_ {(0,r)\times \Sigma} \omega  = 0, \hspace{3mm} \forall r\in (0,\infty).
\end{equation}

\begin{defn}[Symmetric Decreasing Rearrangement of a Function on a Manifold]
    The symmetric decreasing rearrangement of a function \(f:M\rightarrow \mathbb{R}_+\) is the function \(f^*:M\rightarrow \mathbb{R}_+\) defined by,
\end{defn}

\begin{equation}
    f^*(x) = \int_0^{\infty} \mathcal{X}_{\{y:f(y)>t\}^*}(x) \,dt,
\end{equation}
where the set \(\{y:f(y)>t\}^* \subset M\) is well-defined.

We would like to use the concepts of rearrangements of sets and functions to investigate geometric inequalities. Particularly, the isoperimetric, P\'{o}lya-Szeg\H{o} and Sobolev inequalities. It has been shown in \cite{isop_sob} that generalised isoperimetric and Sobolev inequalities are equivalent on specific Riemannian manifolds. Particularly, let \(f\) be a non-negative smooth function with compact support on a Riemannian manifold \(\Sigma^n\) and \(p\in [1,n)\). Let \(q=\frac{np}{n-p}\), then,
\begin{equation}
    ||f||_{L^q(\Sigma)} \leq \frac{(n-1)p}{n-p}||\nabla f||_{L^p(\Sigma)} + ||fH||_{L^p(\Sigma)}.
\end{equation}

Now if \(\Sigma^n\subset \mathbb{R}^{n+m}\) is a compact submanifold and \(\Omega \subset \Sigma\) has sufficiently smooth boundary \(\partial \Omega\), then,
\begin{equation}
    (\text{Vol}(\Omega))^{\frac{n-1}{n}} \leq \text{Area}(\partial \Omega) + ||H||_{L^1(\Omega)}.
\end{equation}

The aim of this paper is to work towards an equivalence between the isoperimetric and P\'{o}lya-Szeg\H{o} inequalities in a similar setting on smooth oriented Riemannian manifolds.

\vspace{7mm}

\textbf{Acknowledgements.} I would particularly like to thank my supervisor Maxwell Stolarski for support throughout the entire process of writing this paper and many helpful discussions of the material.

\section{Preliminaries}\label{sec:pre}

\subsection{Geometric Measure Theory and Integration}

When discussing the co-area formula, we will need to take volumes of lower-dimensional spaces and this is done with the Hausdorff measure.


\begin{defn}[Outer Measure]
    Given an arbitrary set \(X\), an outer measure on \(X\) is a monotone subadditive function, \(\mu:\mathcal{P}(X)\rightarrow [0,\infty]\), satisfying \(\mu(\emptyset) = 0\), where \(\mathcal{P}(X)\) denotes the collection of subsets of \(X\). Hence for countable subsets, \(A,A_1,A_2,\dots\) of \(X\),
    \begin{equation}
        \mu(A) \leq \sum_{j=1}^{\infty} \mu (A_j) \hspace{2mm} \text{when } A\subset \bigcup_{j=1}^{\infty} A_j.
    \end{equation}
    Particularly, this implies \(\mu(A)\leq \mu(B)\) whenever \(A\subset B \subset X\).
\end{defn}

Now suppose \(X\) is a metric space with metric \(d\). Define the diameter of a subset \(A\subset X\) as,
\begin{equation}
    \text{diam}(A) = \underset{x,y\in A}{\text{sup}} d(x,y).
\end{equation}

We can define the \(n\)-dimensional Hausdorff outer measure as follows, noting that we drop the 'outer' later.

\begin{defn}[Hausdorff  pre-measure]
    For any \(n\geq 0\), define the \(n\)-dimensional Hausdorff pre-measure as,
    \begin{equation}
        \mathcal{H}^n_{\delta}(A) = \frac{\omega_n}{2^n} \text{inf} \bigg\{ \sum_{j=1}^{\infty} \big(\text{diam}C_j \big)^n: A\subset \bigcup_{j=1}^{\infty}C_j, \text{diam}C_j < \delta \bigg\}.
    \end{equation}
\end{defn}

\begin{defn}[Hausdorff Measure]
    Then the \(n\)-dimensional Hausdorff measure is defined as,
    \begin{equation}
        \mathcal{H}^n(A) = \underset{\delta > 0}{\text{sup}}\hspace{1mm}\mathcal{H}^n_{\delta}(A) = \underset{\delta \rightarrow 0^+}{\mathcal{H}^n_{\delta}(A)}.
    \end{equation}
\end{defn}

On \(\mathbb{R}^n\), the outer measures \(\mathcal{L}^n, \mathcal{H}^n, \mathcal{H}^n_{\delta}\) all coincide for each \(\delta >0\). This is proven in \cite{gmt}, using Steiner symmetrisation and the isodiametric inequality,
\begin{equation}
    \mathcal{L}^n(A) \leq \frac{\omega_n}{2^n} \big(\text{diam}(A) \big)^n, \hspace{3mm} \forall A \subset \mathbb{R}^n. 
\end{equation}

\subsection{Sobolev Spaces and Lipschitz Continuous Functions}

This section gives context and intuition about the functions and function spaces that will be used in the main theorems of this text, but the rigorous details of the functional analysis is not the primary aim of this paper.

Define the Sobolev space, \(W^{k,p}(\Omega)\) as the subset of functions \(f\in L^p(\Omega)\), such that \(f\) and its weak derivatives up to order \(k\) have finite \(L^p\)-norm, that is
\begin{equation}
    W^{k,p}(\Omega) = \{\phi \in L^p(\Omega):D^{\alpha}u\in L^p(\Omega) \hspace{1mm} \forall |\alpha| \leq k\}.
\end{equation}

Taking \(p=2\), define the Hilbert space, \(H^k = W^{k,2}\), with norm \(||\cdot||_{ W^{k,2}(\Omega)}\).
Define the space, \(H^1_0(\Omega):= \{u\in H^1(\Omega):u|_{\partial \Omega}=0\}\), so for any function \(\phi \in H^1_0(\Omega)\), both \(\nabla \phi\) and \(\Delta \phi\) are well-defined. \(H^1_0(\Omega)\) can be also considered as the completion of \(C_0^{\infty}(\Omega)\), the space of infinitely differentiable functions which are nonzero only on a compact subset of \(\Omega\). That is, \(H^1_0(\Omega)\) is a complete topological vector space that contains a dense vector subspace that is isomorphic to \(C_0^{\infty}(\Omega)\).

We would like to consider Lipschitz functions \(f:\mathbb{R}^n\rightarrow\mathbb{R}^m\), which are differentiable almost everywhere from Rademacher's theorem. Recall that a function is Lipschitz on a metric space \((X,d)\) if there exists a \(K\geq 0\), such that,
\begin{equation}
    d(f(x),f(y)) \leq K\cdot d(x,y), \hspace{3mm} \forall x,y\in X.
\end{equation}
A proof of Rademacher's theorem can be seen in \cite{iso_chavel}.

It is useful to make a brief note of the relationship between Lipschitz continuous functions and the Sobolev spaces. The following theorem in quoted from \cite{sob},

\begin{theorem}
    A function \(u\in L^1_{loc}(\mathbb{R}^n)\) has a representative that is bounded and Lipschitz continuous if and only if \(u\in W^{1,\infty}(\mathbb{R}^n)\).
\end{theorem}

Now if we have a proper open subset \(U\subset \mathbb{R}^n\), if \(u\in W^{1,\infty}(U)\), then we get that any point has a bounded neighbourhood on which \(u\) has a representative that is Lipschitz.

\section{Results in \(\mathbb{R}^n\)}\label{sec:R^n}

Now that we have discussed the measure theory and integration, we can proof rearrangement inequalities on functions in \(\mathbb{R}^n\).

\subsection{Symmetric Decreasing Rearrangement Inequalities on \(\mathbb{R}^n\)}

First we want to prove Cavalieri’s Principle, which gives us the layer-cake decomposition of functions that was stated before. 

\begin{prop}[Cavalieri’s Principle]\label{cav}
    Given a measure space \((\Omega, \mathcal{A},\mu)\), and \(\mathcal{A}\)-measurable function, \(f:\Omega \rightarrow \mathbb{R}_+\), let \(\nu\) be a Borel measure on \([0,\infty)\). Set,
    \begin{equation}
        \phi(t) := \int_0^t d\nu(s) < \infty, \hspace{2mm} \forall t>0.
    \end{equation}
    Then, 
    \begin{equation}
        \int_{\Omega} \phi (f(x)) d\mu(x) = \int_0^{\infty} \mu\big(\{f>t\} \big)d\nu(t).
    \end{equation}
\end{prop}
\begin{proof}
    Applying Fubini's theorem,
    \begin{equation}
        \begin{split}
            \int_0^{\infty} \mu\big(\{f>t\} \big)d\nu(t) & = \int_0^{\infty} \bigg( \int_{\Omega} \mathcal{X}_{\{f>t\}}(x) d\mu(x)\bigg)d\nu(t) 
            \\
            & = \int_{\Omega} \bigg( \int_0^{\infty} \mathcal{X}_{\{f>t\}}(x) d\nu(t)\bigg)d\mu(x)\\
            & = \int_{\Omega} \bigg( \int_0^{f(x)} d\nu(t)\bigg)d\mu(x)\\
            & = \int_{\Omega} \phi (f(x)) d\mu(x).
        \end{split}
    \end{equation}
\end{proof}
    If we choose \(\phi(t)=t, d\nu(t) = dt\) and \(\mu\) as the Dirac delta distribution centred on \(x\), we get the layer-cake decomposition.
    \begin{equation}
        f(x) = \int_0^{\infty} \mathcal{X}_{\{y:f(y)>t\}}(x)dt.
    \end{equation}

For now, unless otherwise mentioned, take \(d\nu(t)=dt\), like was implicitly done in the introduction.

An equivalent way of writing \(f^*\) from the layer-cake decomposition is,

\begin{equation}
    f^*(x) = \text{sup} \{t>0: \text{Vol} (\{f>t\}) \geq \omega _n |x|^n\}.
\end{equation}

Using these two representations of \(f^*\), we can proof the obvious geometric fact that the rearrangement of the level sets of a function is equivalent to level sets of the rearrangement of the function.

\begin{thm}[Rearrangement of level sets is equivalent to level sets of rearrangement]\label{main_thm}
    The two sets are equal, \(\{x\in \mathbb{R}^n: f^*(x) > t\} = \{x\in \mathbb{R}^n: f(x) >t\}^*\)
\end{thm}
\begin{proof}
Write the set \(\{x\in \mathbb{R}^n: f(x) >t\}^*\) as \(\{f>t\}\) for clarity. 
First show \(\{f>t\}^* \subset \{f^* > t\}\) for fixed \(t>0\) and \(y\in \{f>t\}^*\). Then the Lebesgue measure of the centred ball of radius \(y\) is,
\begin{equation}
    \omega _n |y|^n < |\{f>t\}|,
\end{equation}
as \(|\{f>t\}| = |\{f>t\}^*|\) by definition, and \(\{f>t\}^*\) has a finite volume. Write \(f^*\) as a layer-cake decomposition and split into two integrals,
\begin{equation}
    \begin{split}
        f^*(y) & = \int_0^{\infty} \mathcal{X} _{\{x:f(x)>s\}^*} (y) ds\\
        & = \int_0^t \mathcal{X} _{\{x:f(x)>s\}^*} (y) ds + \int_t^{\infty} \mathcal{X} _{\{x:f(x)>s\}^*} (y) ds.
    \end{split}
\end{equation}
For all \(s\in (0,t)\), \(\{f>t\} \subset \{f>s\}\), so \(\{f>t\}^* \subset \{f>s\}^*\) and \(\mathcal{X} _{\{f>s\}^*} (y) = 1\). This means the first integral is just equal to \(t\).

Suppose for a contradiction that \(\int_t^{\infty} \mathcal{X} _{\{x:f(x)>s\}^*} (y) ds = 0\), so \(\forall s\geq t,  y\notin \{f>s\}^*\). Then by Fatou's lemma,
\begin{equation}
    \begin{split}
        \omega _n |y|^n & \geq |\{f>s\}|\\
        & \geq |\liminf_{s\rightarrow t}\{f>s\}|\\
        & = |\{f\geq t\}| \geq |\{f>t\}|,
    \end{split}
\end{equation}
which is a contradiction. Hence, \(\int_t^{\infty} \mathcal{X} _{\{x:f(x)>s\}^*} (y) ds > 0\) and we have,
\begin{equation}
    f^*(y) > t \implies  \{f>t\}^* \subset \{f^* > t\}.
\end{equation}
For the reverse inclusion, fix a  \(y\notin \{f>t\}^*\), so \(\forall s\geq t\),  \(\mathcal{X} _{\{f>s\}^*} (y) = 0\). Hence for any \(w\in \{f>t\}^*\), sup\(\{s>0 : w\in \{f>t\}^*\} \leq t\). Similarly, it then follows \(f^*(y) \leq t\), meaning we have shown,
\begin{equation}
    y\notin \{f>t\}^* \implies y\notin \{f^*>t\}
\end{equation}
or equivalently,
\begin{equation}
    y\in \mathbb{R}^n \setminus \{f>t\}^* \implies y\in \mathbb{R}^n \setminus \{f^*>t\}
\end{equation}
and taking complements completes the proof. 
\end{proof}

Define the distribution function of \(f\) as,
\begin{equation}
    \mu _f (t) = \text{Vol}\left( \{x|f(x)>t\}\right).
\end{equation}
Then, since \(f^*\) is lower semicontinuous, \(\mu_f\) uniquely determines \(f^*\) and we can say that \(f\) vanishes at infinity if,
\begin{equation}
    \mu_f(t)  < \infty, \hspace{3mm} \forall t>0.
\end{equation}

Then from Theorem \ref{main_thm}, we have,
\begin{equation}\label{dist}
    \text{Vol}\left( \{x|f(x)>t\}\right) = \text{Vol}\left( \{x|f(x)>t\}^*\right) = \text{Vol}\left( \{x|f^*(x)>t\}\right),
\end{equation}
or \(\mu_f(t) = \mu_{f^*}(t), \hspace{1mm} \forall t>0\). We say \(f^*\) is equimeasurable with \(f\), and intuitively this means the corresponding level sets of the two functions have equal measure. If \(f\) is sketched as a graph of \(\mathbb{R}^n\) against \(\mathbb{R}\), the process of producing the graph of \(f^*\) is done by choosing a fixed \(t\) and sliding that level set so it is symmetric about the vertical axis and repeating for every \(t\) value, hence producing a radially decreasing symmetric function that is equimeasurable with the original function \(f\). This also provides clear geometric intuition for the previous theorem of super level sets.

We have another corollary for the correspondence of the rearrangement of a characteristic function of a set, with the characteristic function of the rearranged set.

\begin{corollary}[Rearrangement of a characteristic function of a set is the characteristic function of the rearrangement of the set]
For an open subset \(A\subset \mathbb{R}^n\), we have \(\mathcal{X} ^*_A = \mathcal{X} _{A^*}\).
\end{corollary}
    \begin{proof} 
    From the definition of the characteristic function, we can write,
    \begin{equation}
        \{\mathcal{X} ^*_A >t \} = \left\{
        \begin{array}{ll}
              \emptyset & , t\geq 1 \\
              A^* & , t\in [0,1) \\
              \mathbb{R}^n & , t<0 \\
        \end{array} 
        \right.
    \end{equation}
    
    Then using Theorem \ref{main_thm},
    \begin{equation}
        \begin{split}
            \mathcal{X} ^*_A (y) & = \int_0^{\infty} \mathcal{X} _{\{x:\mathcal{X} ^* _A(x)>t\}} (y) dt\\
            & = \int_0^{\infty} \mathcal{X} _{\{x:\mathcal{X} _A(x)>t\}^*} (y) dt\\
            & = \int_0^1 \mathcal{X} _{\{x\in A^*\}} (y) dt\\
            & = \mathcal{X} _{A^*} (y).
        \end{split}
    \end{equation}
    \end{proof}
Now we have the tools to move into the first proper theorem of rearrangements, as we have reduced questions about functions to questions about their level sets.

\begin{lemma}[Symmetric Decreasing Rearrangement Preserves \(L^p\)-norms]
    For all non-negative \(f\in L^p(\mathbb{R}^n)\),
    \begin{equation}
        ||f||_{L^p(\mathbb{R}^n)} = ||f^*||_{L^p(\mathbb{R}^n)}.
    \end{equation}
\end{lemma}
Sometimes we denote \(||\cdot||_{L^p(\mathbb{R}^n)}\) by just \(||\cdot||_p\) when it is clear what domain the norm is taken with respect to.
\begin{proof}
    Using the layer-cake decomposition and Fubini's theorem,
    \begin{equation}
        \begin{split}
            ||f||^p _p & = \int_{\mathbb{R}^n} f(x)^p \,d\mu(x)\\
            & = \int_{\mathbb{R}^n} \int_0^{\infty} \mathcal{X} _{\{f(x)^p > t\}} \,dt \,d\mu(x)\\
            & = \int_0^{\infty} \text{Vol} \left( \{f(x)^p > t\} \right) \,dt,
        \end{split}
    \end{equation}
    and substituting \(t = s^p, dt = ps^{p-1}ds\) gives, 
    \begin{equation}
        ||f||^p _p = \int_0^{\infty} \text{Vol} \left( \{f(x) > s\} \right) ps^{p-1} \,ds
    \end{equation}
    and as \(f\) is equimeasurable with \(f^*\), the result follows. 
\end{proof}

The case of \(p=1\) simply states that the volume under the surfaces \(f,f^*\) are equal, is intuitively clear from Theorem \ref{main_thm} and equation \ref{dist}. 

\begin{prop}[Symmetric Decreasing Rearrangement is Order-Preserving]
\end{prop}
\begin{center}
\(f(x)\leq g(x) \hspace{2mm} \forall x \in \mathbb{R}^n \implies f^*(x)\leq g^*(x) \hspace{2mm} \forall x \in \mathbb{R}^n\).   
\end{center}

\begin{proof}
    Since \(f(x)\leq g(x) \hspace{2mm} \forall x \in \mathbb{R}^n\), we have \(\{y:f(y)>t\}^* \subset \{y:g(y)>t\}^*\). Hence,
    \begin{equation}
        \begin{split}
        f^*(x) & = \int_{\mathbb{R}^n}\mathcal{X}_{\{y:f(y)>t\}^*}d\mu(x)\\
        & \leq \int_{\mathbb{R}^n}\mathcal{X}_{\{y:g(y)>t\}^*}d\mu(x)\\
        & = g^*(x).
        \end{split}
    \end{equation}
\end{proof}

\begin{theorem}[Hardy-Littlewood Inequality]
    Suppose \(f,g: \mathbb{R}^n \rightarrow \mathbb{R}\) are measurable, non-negative functions that vanish at infinity. Then,
\end{theorem}

\begin{equation}
    \int _{\mathbb{R}^n} f(x)g(x) d\mu(x) \leq \int _{\mathbb{R}^n} f^*(x)g^*(x) d\mu(x).
\end{equation}

\begin{proof}
    Consider \(f=\mathcal{X} _A, g = \mathcal{X} _B\) for measurable open subsets, \(A,B\subset \mathbb{R}^n\), then 
    \begin{equation}
        \text{Vol}\left(A^* \cap B^* \right) = \text{min}\{\text{Vol}\left(A^*\right), \text{Vol}\left(B^* \right)\} = \text{min}\{\text{Vol}\left(A\right), \text{Vol}\left( B \right)\} \geq \text{Vol}\left(A \cap B \right).
    \end{equation}
    Now for general \(f,g\), use the Layer cake decomposition,
    \begin{equation}
        \begin{split}
            \int _{\mathbb{R}^n} f(x)g(x) & = \int _{\mathbb{R}^n} \int _0^{\infty}  \int _0^{\infty} \mathcal{X} _{\{f>s\}} (x) \mathcal{X} _{\{g>t\}} (x) \,ds \,dt d\mu(x)\\
            & = \int _0^{\infty} \int _0^{\infty} \text{Vol} \left( \{f>s\} \cap \{g>t\} \right) \,ds \,dt\\
            & \leq \int _0^{\infty} \int _0^{\infty} \text{Vol} \left( \{f^*>s\} \cap \{g^*>t\} \right) \,ds \,dt\\
            & = \int _{\mathbb{R}^n} f^*(x)g^*(x) d\mu(x).
        \end{split}
    \end{equation}
\end{proof}
\begin{lemma}
    For functions \(f,g\) non-negative and summable,
\end{lemma}
\begin{equation}
    \int _{\mathbb{R}^n} f\mathcal{X} _{\{g\leq s\}} d\mu(x) \geq \int _{\mathbb{R}^n} f^*\mathcal{X} _{\{g^*\leq s\}} d\mu(x).
\end{equation}
\begin{proof}
Using \(\mathcal{X} _{\{g\leq s\}} = 1 - \mathcal{X} _{\{g > s\}}\),
    \begin{equation}
        \int _{\mathbb{R}^n} f\mathcal{X} _{\{g\leq s\}} d\mu(x) = \int _{\mathbb{R}^n} f d\mu(x) - \int _{\mathbb{R}^n} f\mathcal{X} _{\{g > s\}} d\mu(x),
    \end{equation}
using the Hardy-Littlewood inequality on both integrals and previous results,
    \begin{equation}
        \begin{split}
            \int _{\mathbb{R}^n} f\mathcal{X} _{\{g\leq s\}} d\mu(x) & \geq \int _{\mathbb{R}^n} f^* d\mu(x) - \int _{\mathbb{R}^n} f^*\mathcal{X} ^* _{\{g > s\}} d\mu(x),\\
            & = \int _{\mathbb{R}^n} f^* d\mu(x) - \int _{\mathbb{R}^n} f^*\mathcal{X} _{\{g^* > s\}} d\mu(x)\\
            & = \int _{\mathbb{R}^n} f^*(1-\mathcal{X} _{\{g^* > s\}}) d\mu(x)\\
            & = \int _{\mathbb{R}^n} f^*\mathcal{X} _{\{g^* \leq s\}} d\mu(x).
        \end{split}
    \end{equation}
\end{proof}

We need the following corollary of Cavalieri’s principle.
    \begin{corollary}\label{decomp}
        \begin{equation}
            ||f(x)||_{L_p(\Omega)}^p := \int_{\Omega} |f(x)|^p d\mu(x) = p \int^{\infty}_0 \mu\big(\{f>t\} \big)t^{p-1} dt.
        \end{equation}
    \end{corollary}   
\begin{proof}
    Set \(d\nu(t) = pt^{p-1}dt\) in \ref{cav}, so \(\phi(t) = t^p\) and we get the above result. 
\end{proof}

\begin{thm}[Rearrangement Decreases \(L^p\) Distance]
    \(||f-g||_p \geq ||f^* - g^*||_p, \hspace{1mm} 1\leq p \leq \infty \hspace{1mm} \forall x\in \mathbb{R}^n\), with equality if and only if \(f\) and \(g\) have the same family of level sets. 
\end{thm}
\begin{proof}
    Write,
    \begin{equation}
        |f(x)-g(x)| = [f(x)-g(x)]_+ + [g(x)-f(x)]_+,
    \end{equation}
    where \([\dots]_+\) denotes the positive part of the inside argument. This is equivalent to writing,
    \begin{equation}
        |f(x)-g(x)| = \underset{t\in \mathbb{N}}{\text{sup}}\bigg[[f(x)-t]_{+} \mathcal{X} _{\{g\leq t\}} + [g(x)-t]_{+} \mathcal{X} _{\{f\leq t\}}\bigg],
    \end{equation}
    and so appealing to \autoref{decomp}, we can write,
    \begin{equation}
        |f(x)-g(x)|^p = p\hspace{1mm} \underset{t\in \mathbb{N}}{\text{sup}}\int _0 ^{\infty} [f(x)-t]_{+}^{p-1} \mathcal{X} _{\{g\leq t\}} + [g(x)-t]_{+}^{p-1} \mathcal{X} _{\{f\leq t\}} \,dt,
    \end{equation}
    and integrating over \(\mathbb{R}^n\),
    \begin{equation}
        \begin{split}
            ||f(x)-g(x)||^p_p & = \int _{\mathbb{R}^n} |f(x)-g(x)|^p d\mu(x),\\
            & = \underset{t\in \mathbb{N}}{\text{sup}}\int _{\mathbb{R}^n} p\int _0 ^{\infty} [f(x)-t]_{+}^{p-1} \mathcal{X} _{\{g\leq t\}} + [g(x)-t]_{+}^{p-1} \mathcal{X} _{\{f\leq t\}} \,dt d\mu(x),
        \end{split}
    \end{equation}
    and applying the previous lemma gives,
    \begin{equation}
        \begin{split}
            ||f(x)-g(x)||^p_p & \geq \underset{t\in \mathbb{N}}{\text{sup}} \int _{\mathbb{R}^n} p\int _0 ^{\infty} [f^*(x)-t]_{+}^{p-1} \mathcal{X} _{\{g^*\leq t\}} + [g^*(x)-t]_{+}^{p-1} \mathcal{X} _{\{f^*\leq t\}} \,dt d\mu(x),\\
            & = ||f^*(x)-g^*(x)||^p_p,
        \end{split}
    \end{equation}
    using that \(\left( [f(x)-t]_{+}^{p-1} \right) ^* = [f^*(x)-t]_{+}^{p-1}\) and the same for \(g\).
\end{proof}
We have equality if and only if \((f(x)-g(x))(f(y)-g(y)) >= 0\) for almost all \(x,y\).

We can derive a stronger result, for a more general centred convex function, than just the \(L^p\)-norm.\cite{analysis}

\begin{thm}[Non-expansivity of Rearrangement]
    Let \(J:\mathbb{R} \rightarrow \mathbb{R}\) be a non-negative convex function with \(J(0)=0\). Then for non-negative functions \(f\) and \(g\) on \(\mathbb{R}^n\), vanishing at infinity
    \begin{equation}
        \int_{\mathbb{R}^n} J\left(f^*(x)-g^*(x) \right) d\mu(x) \leq \int_{\mathbb{R}^n} J\left(f(x)-g(x) \right) d\mu(x).
    \end{equation}
\end{thm}

\begin{proof}
    Split \(J=J_+ + J_-\), where \(J_+\) is the positive part of \(J\) and similar for \(J_-\) and note both parts are still convex. Hence, \(J_+\) has a right derivative \(J_+'(t)\), for all \(t\), such that \(J_+(t) = \int_0^t J_+'(s) \,ds\) and \(J_+'\) is a non-decreasing function of \(t\). So write,
    \begin{equation}
        J_+(f(x)-g(x)) = \int_{f(x)}^{g(x)}J_+'(s)(f(x)-s)\,ds = \int_0^{\infty}J_+'(f(x)-s)\mathcal{X}_{\{g\leq s\}}\,ds.
    \end{equation}
    Integrating over \(\mathbb{R}^n\), using Fubini's theorem and applying Lemma 1.3 gives,
    \begin{equation}
        \begin{split}
            \int_{\mathbb{R}^n} J_+(f(x)-g(x)) d\mu(x) & \geq \int_0^{\infty} \int_{\mathbb{R}^n} (J_+')^*(f(x)-s)\mathcal{X}_{\{g^*\leq s\}}d\mu(x) \,ds\\
            & = \int_{\mathbb{R}^n} \int_{f^*(x)}^{g^*(x)} J_+'(s)(f(x)-s)\,ds d\mu(x)\\
            & = \int_{\mathbb{R}^n} J_+(f^*(x)-g^*(x)) d\mu(x).
        \end{split}
    \end{equation}
    Repeat for \(J_-\), to get,
    \begin{equation}
        \begin{split}
            J_-(f(x)-g(x)) & = \int_{f(x)}^{g(x)}J_-'(s)(s-g(x))\,ds\\
            & = \int_0^{\infty}J_+'(s-g(x))\mathcal{X}_{\{f\leq s\}}\,ds
        \end{split}
    \end{equation}
    and
    \begin{equation}
        \int_{\mathbb{R}^n} J_-(f(x)-g(x)) d\mu(x) \geq \int_{\mathbb{R}^n} J_-(f^*(x)-g^*(x))d\mu(x).
    \end{equation}
    Summing the results we get,
    \begin{equation}
        \begin{split}
            \int_{\mathbb{R}^n} J(f(x)-g(x)) d\mu(x) & \geq \int_{\mathbb{R}^n} J_+(f^*(x)-g^*(x)) + J_-(f^*(x)-g^*(x)) d\mu(x)\\
            & = \int_{\mathbb{R}^n} J(f^*(x)-g^*(x)) d\mu(x).
        \end{split}
    \end{equation}
\end{proof}

Clearly if we take \(J(x)=|x|^p\), we attain the previous theorem.

\begin{defn}[Convolution]
    Define the convolution of two integrable functions, \(f,g\) by,
\end{defn}

\begin{equation}
    \left( f*g \right)(x) := \int _{-\infty}^{\infty} f(x-y)g(y) \,dy.
\end{equation}

\begin{thm}[Riesz' Rearrangement Inequality]\label{RRT}
    For measurable, non-negative functions, \(f,g,h:\mathbb{R}^n \rightarrow \mathbb{R}\), the functional,
\end{thm}

\begin{equation}
    I(f,g,h) := \int _{\mathbb{R}^n} f(x) g*h(x) d\mu(x),
\end{equation}
can only increase under symmetric decreasing rearrangement. That is, 
\begin{equation}
    I(f,g,h) \leq I(f^*,g^*,h^*).
\end{equation}
This is proven in Lieb and Loss \cite{analysis}, in two ways, using a compactness argument and a symmetry argument.

This states that the convolution of two functions is dominated by the convolution of their rearrangements, in the sense,

\begin{equation}
    \int _B (g*h)^* (x) d\mu(x) = \sup_{C:\text{Vol}(C)=\text{Vol}(B)} \int_C g*h(x)d\mu(x) \leq \int _B g^* * h^* (x) d\mu(x),
\end{equation}
for any centred ball \(B\). Some applications only require a simpler case with \(h(x) = H(|x|)\), such as the Coulomb kernel \(|x-y|^{-1}\) on \(\mathbb{R}^3\) or the heat kernel \((4\pi t)^{-n/2} e^{-\frac{|x-y|^2}{4t}}\), which reduces to the Hardy-Littlewood inequality in the limit \(t\rightarrow 0\).

\subsection{Area, Co-area Formulas and Perimeter}

With the theory of the Hausdorff measure discussed earlier, we can use this to proof the co-area formula. This splits an \(n\)-dimensional integral into an \((n-1)\)-dimensional integral over level sets and a \(1\)-dimensional integral which parameterise the level sets. The simplest example of this is using polar coordinates and writing the \(3\)-dimensional volume integral of a 2-sphere as the product of a 2-dimensional surface integral for a concentric infinitesimally thin shell and a 1-dimensional integral over all the possible radii of the sphere.

This gives a formal definition of the perimeter of an \(n\)-dimensional object and we can use this powerful formula to prove the isoperimetric and P\'{o}lya-Szeg\H{o} inequalities in \(\mathbb{R}^n\).

First recall some linear algebra, specifically the Jacobian of a linear operator. Given a linear operator, \(T:\mathbb{R}^n\rightarrow \mathbb{R}^m\), where \(m\geq n\), the \(n\)-dimensional Jacobian of \(T\) is defined as,
\begin{equation}
    \mathcal{J}_nT = \sqrt{\text{det}(T^*\circ T)},
\end{equation}
where \(T^*:\mathbb{R}^m\rightarrow \mathbb{R}^n\) is the adjoint map of \(T\). 

Now if we have a Lipschitz function \(f:\mathbb{R}^n\rightarrow \mathbb{R}^m\), then if \(f\) is differentiable at \(x\), denote the differential at \(x\) by \(df_x\), which is a linear operator from \(\mathbb{R}^n\) to \(\mathbb{R}^m\).

\begin{thm}[Area Formula]
    Let \(f:\mathbb{R}^n\rightarrow \mathbb{R}^m\) be a Lipschitz continuous function, or \(f\in W^{1,\infty}(\mathbb{R}^n)\). Then, for any Borel subset \(A\subset \mathbb{R}^n\), the function,
    \begin{equation}
        \mathbb{R}^m \ni y \mapsto \mathcal{H}^0(A\cap f^{-1}(y)),
    \end{equation}
    is \(\mathcal{H}^n\)-measurable. Then,
    \begin{equation}
        \int_{\mathbb{R}^m} \mathcal{H}^0(A\cap f^{-1}(y)) d\mathcal{H}^n(y) = \int_A \mathcal{J}_n df_x dx.
    \end{equation}
\end{thm}

If the rank of \(df_x < n\) or \(f\) is not differentiable at \(x\), then \(\mathcal{J}_n\) vanishes. If \(E\) is a set of these such \(x\), then \(\mathcal{H}^n(f(E))=0\). This can be rephrased as the set of critical points of \(f\) has measure zero in \(\mathbb{R}^m\), which is a version of Sard's theorem and discussed in the manifolds context later.

A simple case of the area formula is when \(m=1\) is just regular polar coordinates,

\begin{thm}[Polar Coordinates]
    Let \(f: \mathbb{R}^n \rightarrow \mathbb{R}\) be as before. For each \(x_0 \in \mathbb{R}^n\),
    \begin{equation}
        \int _{\mathbb{R}^n} f(x)\,dx = \int _0^{\infty} \bigg( \int _{\partial B_r (x_0)} f(x) \,dS(x) \bigg) \,dr,
    \end{equation}
    where \(S\) denotes the surface measure on the boundary of  \(B_r(x_0)\).
\end{thm}
It also produces the classical formula for the length of a parameterized curve, \(\phi(t):\mathbb{R}\rightarrow \mathbb{R}^n\), when \(\phi\) is injective,
\begin{equation}
    \mathcal{H}^1(\phi(A)) = \int_A |\phi'(t)| dt.
\end{equation}

Now consider the co-area formula. 

\begin{theorem}[Co-Area Formula]
    Let \(f:\mathbb{R}^n\rightarrow \mathbb{R}^m\) be a Lipschitz continuous function, or \(f\in W^{1,\infty}(\mathbb{R}^n)\), and \(g:\mathbb{R}^n\rightarrow [0,\infty]\) a Borel function. Then,
    \begin{equation}
        \mathbb{R}^m \ni y \mapsto \int_{\{f = y\}} g(x) d\mathcal{H}^0(x),
    \end{equation}
    is a \(\mathcal{H}^n\)- measurable function and the co-area formula states,
    \begin{equation}
        \int_{\mathbb{R}^m} \bigg(\int_{\{f=y\}} g(x) d\mathcal{H}^0(x) \bigg) d\mathcal{H}^n(x) = \int_{\mathbb{R}^n} g(x) \mathcal{J}_ndf_x dx.
    \end{equation}
\end{theorem}
For now we consider \(f\) when \(m=1\) and use the parameter \(t\). In this case it is convenient to write \(\mathcal{J}_ndf_x\) as \(\nabla f\). Then the co-area formula reads as,
\begin{equation}
    \int _{\mathbb{R}^n} g(x) |\nabla f(x)| \,dx = \int _0^{\infty} \int _{\{f=t\}} g(x) \,d\mathcal{H}^{n-1} \,dt,
\end{equation}
where the level set \(\{x\in \mathbb{R}^n: f(x)=t\}\) is a smooth, \((n-1)\)-dimensional hypersurface in \(\mathbb{R}^n\).

\begin{defn}[Perimeter of Super-level Sets]
    Define the perimeter of a super level set \(\{f>t\}\) as the surface integral,
\end{defn}

\begin{equation}
    \text{Per}\left( \{f>t\}\right) := \int _{\{f=t\}} \,d\mathcal{H}^{n-1} = \mathcal{H}^{n-1}\big(\{f=t\}\big),
\end{equation}
noting the boundary of the super level set is \(\partial \big(\{f>t\}\big) = \{f=t\}\). The perimeter of a set can also be thought of as the total variation of its characteristic function in an open set containing your set.

Generally for an open subset \(A\subset \mathbb{R}^n\),
\begin{equation}
    \text{Per}(A) := \int_{\partial A} d\mathcal{H}^{n-1} = \mathcal{H}^{n-1}(\partial A).
\end{equation}

Considering different choices of \(g(x)\) can give geometric insights into \(f(x)\). Taking \(g\equiv 1\), we immediately get the following corollary,

\begin{corollary}[\(L^1\)-norm of gradient of \(f\)]
    With \(f\) as above,
\end{corollary}

\begin{equation}
    || \nabla f||_1 = \int _0 ^{\infty} \text{Per}\left( \{f>t\}\right).
\end{equation}

When \(f\) is smooth and where \(\nabla f\) does not vanish, the co-area formula defines a local change of coordinates. Assuming only that \(f\) is smooth, taking \(g = |\nabla f|^{-1}\) gives,

\begin{lemma}\label{crit_pts}
    For any interval \((t_1,t_2] \subset \mathbb{R}\),
\end{lemma}

\begin{equation}
    \int_{t_1}^{t_2} \int _{\{f=t\}} |\nabla f|^{-1} \,d\mathcal{H}^{n-1} \,dt = \text{Vol} \left( \{x: t_1 < f(x) \leq t_2, |\nabla f(x)| \neq 0\} \right),
\end{equation}
where the right hand side is the volume of all the points that are not critical points of \(f\). 
\begin{proof}
    Write, 
    \begin{equation}
        |\nabla f(x)|^{-1} = \lim_{\epsilon \rightarrow 0_+} \left( \epsilon + |\nabla f(x)|\right)^{-1}.
    \end{equation}
    Since \(\epsilon >0\), the sequence of measurable functions, \(g_{\epsilon} = \epsilon + |\nabla f(x)|\) satisfies, \(0\leq g_1 \leq g_2 \leq \dots\) where \(\lim_{\epsilon \rightarrow 0_+} \left( \epsilon + |\nabla f(x)|\right)^{-1} = |\nabla f(x)|^{-1}\), so by the monotone convergence theorem we can interchange the limit and integral. For any points where \(|\nabla f| \neq 0\), \(\epsilon >0\) ensures the inner integral over \(\sigma\) is now well-defined. Then,
    \begin{equation}
        \int_{t_1}^{t_2} \int _{\{f=t\}} \lim_{\epsilon \rightarrow 0_+} \left( \epsilon + |\nabla f(x)|\right)^{-1} \,d\mathcal{H}^{n-1} \,dt = \lim_{\epsilon \rightarrow 0_+} \int_{t_1}^{t_2} \int _{\{f=t\}} \left( \epsilon + |\nabla f(x)|\right)^{-1} \,d\mathcal{H}^{n-1} \,dt,
    \end{equation}
    by the co-area formula,
    \begin{equation}
        = \lim_{\epsilon \rightarrow 0_+} \int _{f^{-1}(t_1)}^{f^{-1}(t_2)} \left( \epsilon + |\nabla f(x)|\right)^{-1} |\nabla f(x)| \,dx = \int _{f^{-1}(t_1)}^{f^{-1}(t_2)} \mathcal{X}_{\{y:|\nabla f(y)|\neq 0\}} (x) \,dx,
    \end{equation}
    taking the limit inside and by definition of volume we are done.
\end{proof}
Note that symmetrically rearranging \(f\) will decrease the number of critical points, i.e.,
\begin{corollary}[Symmetric Rearrangement of a Function Decrease the Number of Critical Points]\label{decrease_crit_pts}
\end{corollary}

\begin{equation}
    \int _{\{f=t\}} |\nabla f|^{-1} \,d\mathcal{H}^{n-1} \leq \int _{\{f^*=t\}} |\nabla f^*|^{-1} \,d\mathcal{H}^{n-1}.
\end{equation}
\begin{proof}
    Since \(f^*\) is a monotone decreasing, non-negative function, it only has a critical point when \(f(x)=0\), i.e. where \(f^*\) crosses the vertical axis. Hence it follows from Lemma \ref{crit_pts}.
\end{proof}

Now we can state the two main geometric inequalities we are interested in.

\begin{lemma}[Isoperimetric Inequality]
    Given a measurable set of finite volume, \(A\subset \mathbb{R}^n\),
    \begin{equation}
        \text{Per}(A) \geq \text{Per}(A^*).
    \end{equation}
\end{lemma}
In 2-dimensions, this means the circle minimises perimeter for a closed shape of fixed area. 

\begin{lemma}[P\'{o}lya-Szeg\H{o} Inequality]
    For a non-negative, measurable function \(f:\Omega \rightarrow \mathbb{R}^n\), where \(\Omega\) is an open subset of \(\mathbb{R}^n\) , with \(f\in L^p(\Omega)\),
\end{lemma}

\begin{equation}
    ||\nabla f||_p \geq ||\nabla f^*||_p.
\end{equation}
\begin{proof}
    By the co-area formula,
    \begin{equation}
        ||\nabla f||_p^p = \int _{\mathbb{R}^n} |\nabla f(x)|^p \,dx = \int_{0}^{\infty} \int _{\{f=t\}} |\nabla f|^{p-1} \,d\mathcal{H}^{n-1} \,dt,
    \end{equation}
    and by Jensen's inequality,
    \begin{equation}
        \int _{\{f=t\}} |\nabla f|^{p-1} \frac{\,d\mathcal{H}^{n-1}}{\text{Per}(\{f>t\})} \geq \bigg( \int _{\{f=t\}} |\nabla f|^{-1} \frac{\,d\mathcal{H}^{n-1}}{\text{Per}(\{f>t\})} \bigg)^{-(p-1)}.
    \end{equation}
    Multiply by Per\((\{f>t\})\), apply Corollary \ref{decrease_crit_pts}, where the inequality is flipped due to the negative exponent, and use the isoperimetric inequality to get,
    \begin{equation}
        \begin{split}
            \int _{\{f=t\}} |\nabla f|^{p-1} \,d\mathcal{H}^{n-1} & \geq \text{Per}\left( \{f>t\}\right)^p \bigg( \int _{\{f=t\}} |\nabla f|^{-1} \,d\mathcal{H}^{n-1} \bigg)^{-(p-1)}\\
        & \geq \text{Per}\left( \{f^*>t\}\right)^p \bigg( \int _{\{f^*=t\}} |\nabla f^*|^{-1} \,d\mathcal{H}^{n-1} \bigg)^{-(p-1)}\\
        & = \int _{\{f^*=t\}} |\nabla f^*|^{p-1} \,d\mathcal{H}^{n-1}.
        \end{split}
    \end{equation}
    where we have equality in Jensen's inequality when \(f=f^*\), since \(|\nabla f^*|\) is constant on the level surface. That is,
    \begin{equation}
        \int _{\{f=t\}} |\nabla f|^{p-1} \,d\mathcal{H}^{n-1} \geq \int _{\{f^*=t\}} |\nabla f^*|^{p-1} \,d\mathcal{H}^{n-1} \hspace{2mm} \text{ or } \hspace{2mm} ||\nabla f||_p \geq ||\nabla f^*||_p,
    \end{equation}
    by integrating over \(t\) and using the co-area formula.
\end{proof}
Note that the \(p=1\) case is close to the isoperimetric inequality, but with integrals of the perimeter.

Another proof of the P\'{o}lya-Szeg\H{o} inequality can be seen in \cite{short_course_RI} using polarization.

Next consider an inequality that will produce the sharp isoperimetric inequality.
\begin{lemma}[Brunn-Minkowski Inequality]
    For two sets, \(B,C \subset \mathbb{R}^n\), define the Minkowski sum by \(B+C = \{b+c : b\in B, c\in C\}\). Then,
\end{lemma}
\begin{equation}
    \text{Vol}(B+C)^{1/n} \geq \text{Vol}(B)^{1/n} + \text{Vol}(C)^{1/n},
\end{equation}
given \(B\) and \(C\) are non-negative measurable sets of finite volume, with equality when \(B\) and \(C\) lie in parallel hyperplanes or are homothetic. 


\begin{lemma}[Sharp Isoperimetric Inequality]
\end{lemma}

\begin{equation}
    \text{Per}(A) \geq C(n) \left( \text{Vol} (A) \right) ^{\alpha (n)},
\end{equation}
where,
\begin{equation}
    C(n) = n \omega _n^{1/n} , \hspace{0.3cm} \alpha _n = (n-1)/n,
\end{equation}
with equality when \(A\) is the unit sphere in \(\mathbb{R}^n\).
\begin{proof}
    Consider the Brunn-Minkowski inequality between a set \(A\) and a ball of radius \(\epsilon\), that is \(B_{\epsilon} = \epsilon B_1\),
    \begin{equation}
        \text{Vol}(A+B_{\epsilon})^{1/n} \geq \text{Vol}(A)^{1/n} + \text{Vol}(B_{\epsilon})^{1/n} = \text{Vol}(A)^{1/n} + (\epsilon \omega _n)^{1/n}.
    \end{equation}
    Raise each side to the power of \(n\),
    \begin{equation}
        \text{Vol}(A+B_{\epsilon}) \geq \bigg( \text{Vol}(A)^{1/n} + (\epsilon \omega _n)^{1/n} \bigg)^n,
    \end{equation}
    subtract Vol(\(A\)), divide by \(\epsilon\) and take the limit \(\epsilon \rightarrow 0\),
    \begin{equation}
        \begin{split}
            \text{Per}(A) & = \lim _{\epsilon \rightarrow 0} \frac{\big( \text{Vol}(A+B_{\epsilon}) -\text{Vol}(A) \big)}{\epsilon}\\
            & \geq \lim _{\epsilon \rightarrow 0} \frac{\bigg( \text{Vol}(A)^{1/n} + (\epsilon \omega _n)^{1/n} \bigg)^n -\text{Vol}(A)}{\epsilon}\\
            & = \text{Vol}(A) \lim _{\epsilon \rightarrow 0} \frac{\bigg( 1 + \bigg[\frac{\epsilon \omega _n}{\text{Vol}(A)}\bigg]^{1/n} \bigg)^n -1}{\epsilon}\\
            & = n \omega _n^{1/n} \text{Vol}(A)^{1-1/n},
        \end{split}
    \end{equation}
    using L'H\^opital's rule for evaluating the final limit.
    When \(A\) is the unit sphere \(S\), Vol\((A) = \omega _n\), so the RHS equals \(n\omega _n = \text{Per}(S)\).
\end{proof}

It is also standard to right the analytic isoperimetric inequality for any bounded domain of \(\mathbb{R}^n, n\geq 2\), as,
\begin{equation}
    \frac{A(\partial \Omega)}{V(\Omega)^{1-1/n}} \geq \frac{A(\mathbb{S}^{n-1})}{V(\mathbb{B}^n)^{1-1/n}},
\end{equation}
where \(V\) denotes the \(n\)-measure, \(A\) denotes the \((n-1)\)-measure, \(\mathbb{B}^n\) is the unit disk in \(\mathbb{R}^n\) and \(\mathbb{S}^{n-1}\) the unit sphere in \(\mathbb{R}^n\). For bounded domains of the plane (\(n=2\)), this is written,
\begin{equation}
    L^2 \geq 4\pi A,
\end{equation}
with \(A\) the area of the domain and \(L\) the length of the boundary. A standard proof \cite{iso_chavel} of the plane case with complex numbers follows.
\begin{proof}
    Let \(\Omega\subset \mathbb{R}^2\) be a relatively compact domain, with connected boundary \(\partial \Omega\in C^1\). Then denote \(z\in \mathbb{C}\) by \(z= x+iy\) and the area measure as an oriented volume element,
    \begin{equation}
        dA = dx \wedge dy = \frac{i}{2} dz \wedge d\bar{z}.
    \end{equation}
    The result follows immediately from Greens' theorem and noting the winding number of \(\partial \Omega\) about any \(z\in \mathbb{C}\) equals \(1\).
    \begin{equation}
        \begin{split}
        4\pi A(\Omega) & = \int \int_{\Omega} 2\pi i dz \wedge d \bar{z},\\
        & = \int \int_{\Omega} dz \wedge d \bar{z} \int _{\partial \Omega} \frac{d\xi}{\xi - z},\\
        & = \int_{\partial \Omega} d\xi \int \int_{\Omega} \frac{dz \wedge d \bar{z}}{\xi - z},\\
        & = \int_{\partial \Omega} d\xi \int_{\partial \Omega}\frac{\bar{\xi}-\bar{z}}{\xi - z} dz \leq L^2(\partial \Omega).
        \end{split}
    \end{equation}
\end{proof}

The isoperimetric inequality is equivalent to the Sobolev inequality in this setting \cite{short_course_RI}. When \(f\in W^{1,p}(\mathbb{R}^n)\), then

\begin{equation}
    C||f||_{L^{p^*}(\mathbb{R}^n)} \leq ||\nabla f||_{L^p({\mathbb{R}^n})},
\end{equation}
where \(1\leq p<n, p^* = \frac{np}{n-p}\). The sharp constant is given by,
\begin{equation}
    C_{Sobolev}(n,p) = \underset{||f||_{p^*}=1}{\text{inf}}||\nabla f||_p.
\end{equation}

\subsection{Application to Poisson's Equation}

Rearrangement inequalities have a vast amount of applications to PDE theory and functional analysis. We will only briefly touch on these here, particularly looking at Poisson's equation.  There is a much greater discussion in the later chapters of \cite{iso_chavel}, of analytic isoperimetric inequalities and the Laplace and Heat operators.

First we will state Talenti's inequality.

\begin{thm}[Talenti's Inequality]
    Let \(f\) be a non-negative smooth function with compact support in \(\mathbb{R}^n\) and \(u,v\) the unique solutions of,
    \begin{equation}\label{talenti}
        -\Delta u = f, \hspace{0.4cm} -\Delta v = f^*,
    \end{equation}
    then,
\end{thm}

\begin{equation}
    u^*(x) \leq v(x).
\end{equation}

This is proven in chapter 4 of \cite{short_course_RI}. 

We can define the electrostatic potential \(u\) associated with \(f\) by,
\begin{equation}
    u(x) = C(n) \int _{\mathbb{R}^n} |x-y|^{-(n-2)} f(y) \,dy,
\end{equation}
and similarly for \(v\),
\begin{equation}
    v(x) = C(n) \int _{\mathbb{R}^n} |x-y|^{-(n-2)} f^*(y) \,dy.
\end{equation}

\begin{corollary}[Potential of Rearrangement Dominates Rearrangement of Potential]
    For any ball \(B\),
\end{corollary}

\begin{equation}
    \int _B u^*(x) \,dx \leq \int _B v(x) \,dx 
\end{equation}
\begin{proof}
    \begin{equation}
        \begin{split}
            \int _B u^*(x) \,dx & = \int _B C(n) \bigg( \int _{\mathbb{R}^n} |x-y|^{-(n-2)} f(y) \,dy \bigg)^* \,dx,\\
            & \leq \int _B C(n) \bigg( \int _{\mathbb{R}^n} |x-y|^{-(n-2)} f(y) \,dy \bigg) \,dx,\\
            & = C(n) \int _{\mathbb{R}^n} |x|^{-(n-2)} * f(x) \,dx\\
            & \leq C(n) \int _{\mathbb{R}^n} |x|^{-(n-2)} * f^*(x) \,dx\\
            & = \int _B v(x) \,dx,
        \end{split}
    \end{equation}
    by Riesz's inequality. 
\end{proof}
\begin{corollary}[Alternative Proof of \(p=2\) case of the P\'{o}lya-Szeg\H{o} Inequality]
\end{corollary}
\begin{proof}
    Let \(f\) be a smooth non-negative function with compact support in \(\mathbb{R}^n\) and define the function,
    \begin{equation}
        u(t,x) = \frac{1}{4\pi t} \int _{\mathbb{R}^n} e^{-\frac{|x-y|^2}{4t}}f(y) \,dy.
    \end{equation}
    \(u\) satisfies the heat equation \(\partial _t u = \Delta u\), with initial values, \(u(0,x) = f(x)\). Define the functional,
    \begin{equation}
        I(t) = \int u(t,x)f(x) \,dx,
    \end{equation}
    and compute,
    \begin{equation}
        \begin{split}
            \partial _t I(t) & = \int (\partial _t u) f\\
            & = \int f \Delta u\\
            & = \int u \Delta f,
        \end{split}
    \end{equation}
    using integration by parts and noting the boundary terms are zero. Take \(t=0\),
    \begin{equation}
        \begin{split}
            \partial _t I(t) |_{t=0} & = \int u(0,x) \Delta f\\
            & = \int f \Delta f\\
            & = -\int \nabla f \cdot \nabla f\\
            & = -||\nabla f||_2^2,
        \end{split}
    \end{equation}
    again using integration by parts. If we define \(J(t)\), by replacing \(f\) with \(f^*\) in the definition of \(I(t)\), the problem reduces to proving,
    \begin{equation}
        \partial _t I(t) |_{t=0} \leq \partial _t J(t) |_{t=0}.
    \end{equation}
    This follows immediately from Riezs' rearrangement inequality,
    \begin{equation}
        \begin{split}
            \partial _t I(t) |_{t=0} & = \frac{1}{4\pi t} \int \int \Delta f(x) e^{-\frac{|x-y|^2}{4t}} f(y) \,dy \,dx\\
            & = \frac{1}{4\pi t} \int \Delta f(x) e^{-\frac{|x|^2}{4t}} * f(x) \,dx\\
            & \leq \frac{1}{4\pi t} \int \Delta f^*(x) e^{-\frac{|x|^2}{4t}} * f^*(x) \,dx\\
            & = \partial _t J(t) |_{t=0},
        \end{split}
    \end{equation}
    using that \(e^{x^2}\) is spherically symmetric and \((\Delta f(x))^* = \Delta f^*(x)\).
\end{proof}

\begin{corollary}
    Let \(f\) be a smooth non-negative function with compact support in \(\mathbb{R}^n\) and let \(u,v\) solve Laplace's equations as in Talenti's inequality. Then,
\end{corollary}

\begin{equation}
    ||\nabla u||_2 \leq ||\nabla v||_2.
\end{equation}

\begin{proof}
    By definition we have,
    \begin{equation}
        \begin{split}
            ||\nabla u||_2^2 & = \int|\nabla u|^2\\
            & = - \int u\Delta u\\
            & = \int uf,
        \end{split}
    \end{equation}
    and,
    \begin{equation}
        ||\nabla v||_2^2 = \int vf^*.
    \end{equation}
    Applying Hardy-Littlewood inequality and then Talenti's inequality,
    \begin{equation}
        \begin{split}
            ||\nabla u||_2^2 & = \int uf\\
            & \leq \int u^* f^*\\
            & \leq \int vf^*\\
            & = ||\nabla v||_2^2.
        \end{split}
    \end{equation}
\end{proof}

\begin{thm}[Faber-Krahn Inequality]
    Let \(\Omega\) be an open set of finite volume in \(\mathbb{R}^n\) and \(\lambda _1 (\Omega)\) be the principal eigenvalue of the Dirichlet Laplacian on \(\Omega\). That is, the smallest value of \(\lambda\) for which the problem,
    \begin{equation}
        \left\{
        \begin{array}{ll}
              \Delta u = \lambda u & , \text{in } \Omega \\
              u=0 & , \text{on } \partial \Omega \\
        \end{array} 
        \right.
    \end{equation}
    has a non-trivial solution, for a given \(u\in H^1_0(\Omega)\). Then,
    \begin{equation}
        \lambda _1(\Omega) \geq \lambda _1 (\Omega ^*).
    \end{equation}
\end{thm}
\begin{proof}
    Follows straight from the Rayleigh-Ritz principle and the P\'{o}lya-Szeg\H{o} inequality. The Rayleigh-Ritz principle states, for \(\phi \in H^1_0(\Omega)\),
    \begin{equation}
        \lambda _1(\Omega) = \inf _{||\phi||_{L^2(\Omega)}=1} \int _{\Omega} |\nabla \phi|^2.
    \end{equation}
    Take \(\phi _1\) as the non-negative normalized minimizing eigenfunction corresponding to \(\lambda_1\), so \(\phi_1 \in  H^1_0(\Omega)\). Hence we can apply the P\'{o}lya-Szeg\H{o} inequality,
    \begin{equation}
        \begin{split}
            \lambda _1(\Omega) & = \int _{\Omega} |\nabla \phi_1|^2\\
            & \geq \int _{\Omega ^*} |\nabla \phi ^*_1|^2\\
            & \geq \inf _{||\phi||_{L^2(\Omega ^*)}=1} \int _{\Omega ^*} |\nabla \phi|^2\\
            & = \lambda _1 (\Omega ^*).
        \end{split}
    \end{equation}
\end{proof}

\subsection{3 Different Proofs of the Isoperimetric Inequality}\label{approx_proofs}

By approximating the perimeter of our surface in \(\mathbb{R}^n\) in different ways, we can deduce the isoperimetric inequality using previous theorems. 

\textbf{(1)} Let \(\phi_{\delta}: \mathbb{R}^n \rightarrow \mathbb{R}\) be a non-negative smooth function, increasing from 0 to 1 over a small strip of width \(\delta\), and approximate the perimeter by,
\begin{equation}
    \text{Per}(A) \approx \int_A |\nabla \phi_{\delta}| \,dx.
\end{equation}

This approximation explicitly means,
\begin{equation}
    \text{Per}(A) = \underset{\delta \rightarrow 0^+}{\text{lim}}\int |\nabla \phi_{\delta}| \,dx.
\end{equation}
\begin{proof}
    Assuming the P\'{o}lya-Szeg\H{o} Inequality,
    \begin{equation}
        \begin{split}
            \text{Per}(A) & \approx \int |\nabla \phi_{\delta}| \,dx\\
            & = ||\nabla \phi_{\delta}||_1\\
            & \geq ||\nabla \phi^*_{\delta}||_1\\
            &=  \int |\nabla \phi^*_{\delta}| \,dx\\
            & \approx \text{Per}(A^*).
        \end{split}
    \end{equation}
    Taking limits gives the exact result.
\end{proof}

\textbf{(2)} Let \(B_{\delta}\) be the centered ball of radius \(\delta\), and approximate the perimeter by,
\begin{equation}
    \text{Per}(A) \approx \frac{1}{\delta}\text{Vol}\left((A+B_{\delta}) \setminus A \right).
\end{equation}

\begin{proof}
    Assuming the Brunn-Minkowski inequality and noting Vol\((A)\) = Vol\(A^*)\),
    \begin{equation}
        \begin{split}
            \text{Vol}(A+B_{\delta})^{1/n} & \geq \text{Vol}(A)^{1/n} + \text{Vol}(B_{\delta})^{1/n}\\
            & = \text{Vol}(A^*)^{1/n} + \text{Vol}(B_{\delta})^{1/n}
        \end{split}
    \end{equation}
    then with equality in the Brunn-Minkowski inequality,
    \begin{equation}
        = \text{Vol}(A^*+B_{\delta})^{1/n}.
    \end{equation}
    Taking the n-th power of each side and subtracting Vol\((A)\) from either side,
    \begin{equation}
        \text{Vol}(A+B_{\delta}) - \text{Vol}(A) \geq \text{Vol}(A^*+B_{\delta}) - \text{Vol}(A^*).
    \end{equation}
    But,
    \begin{equation}
        \begin{split}
            \text{Per}(A) & \approx \frac{1}{\delta}\text{Vol}\left((A+B_{\delta}) \setminus A \right)\\
            & = \frac{1}{\delta}\bigg[\text{Vol}(A+B_{\delta}) - \text{Vol}(A)\bigg],\\
            & \geq \frac{1}{\delta}\bigg[\text{Vol}(A^*+B_{\delta}) - \text{Vol}(A^*)\bigg]\\
            & \approx \text{Per}(A^*).
        \end{split}
    \end{equation}
\end{proof}
Note, taking \(\delta \rightarrow 0\) recovers the previous limit definition of perimeter for deriving the sharp form of the Isoperimetric inequality.

\textbf{(3)} Approximate the perimeter by,
\begin{equation}
    \text{Per}(A) \approx \frac{C(n)}{\delta ^{n+1}} \int _A \mathcal{X}_{A^c} * \mathcal{X} _{B_{\delta}} (x) \,dx,
\end{equation}
where \(C(n) = \frac{n+1}{\omega _{n-1}}\).

\begin{proof}
    Assuming Riesz' inequality, and writing \(\mathcal{X}_{A^c} = 1 - \mathcal{X}_{A}\),
    \begin{equation}
        \begin{split}
            \text{Per}(A) & \approx \frac{C(n)}{\delta ^{n+1}} \int _A \mathcal{X}_{A^c} * \mathcal{X} _{B_{\delta}} (x) \,dx\\
            & = \frac{C(n)}{\delta ^{n+1}} \bigg[ \int_A \mathcal{X} _{B_{\delta}}(x) \,dx - \int _A \mathcal{X}_{A} * \mathcal{X} _{B_{\delta}} (x) \,dx \bigg],\\
            & \geq \frac{C(n)}{\delta ^{n+1}} \bigg[ \int_A \mathcal{X} _{B_{\delta}}(x) \,dx - \int _A \mathcal{X}_{A}^* * \mathcal{X}_{B_{\delta}}^* (x) \,dx \bigg],
        \end{split}
    \end{equation}
    then by results for rearrangements of characteristic functions, and as \((B_{\delta})^* = B_{\delta}\),
    \begin{equation}
        \begin{split}
            \text{Per}(A) & \geq \frac{C(n)}{\delta ^{n+1}} \bigg[ \int_A \mathcal{X} _{B_{\delta}}(x) \,dx - \int _A \mathcal{X}_{A^*} * \mathcal{X}_{B_{\delta}} (x) \,dx \bigg],\\
            & = \frac{C(n)}{\delta ^{n+1}} \int _A \mathcal{X}_{(A^*)^c} * \mathcal{X} _{B_{\delta}} (x) \,dx\\
            & \approx \text{Per}(A^*).
        \end{split}
    \end{equation}
\end{proof}

\section{Symmetric Rearrangement on Smooth Oriented Manifolds \(M^n = (0,\infty)\times \Sigma^{n-1}\)}\label{sec:mflds}

Now we need to build up some theory of Riemannian manifolds, so we can adapt symmetric rearrangement inequalities onto smooth oriented Riemannian manifolds \((M^n,g)\).

\subsection{Review of Tensors and Riemannian Metrics}

\subsubsection{Tensor Review}

Let \(V\) be a finite-dimensional vector space and \(V^*\) the dual space of covectors. Denote elements, \(\omega^i \in V^*, X_j\in V\).

\begin{defn}[\({k\choose l}\)-tensor]
    A \({k\choose l}\)-tensor is a multilinear map of the form,
\begin{equation}
    F: V^*\times \dots \times V^* \times V \times \dots \times V \rightarrow \mathbb{R},
\end{equation}
with \(k\)-copies of \(V\) and \(l\)-copies of \(V^*\). It is also called a \(k-\)covariant, \(l-\)contravariant tensor. The space of mixed \({k\choose l}\)-tensors is denoted \(T^k_l(V)\).
\end{defn}

\begin{defn}[Tensor Product]
    If \(F\in T^k_l(V)\) and \(G\in T^p_q(V)\), the tensor \(F\otimes G \in T^{k+p}_{l+q}(V)\) is called the tensor product of \(F\) and \(G\), defined by,
    \begin{equation}
        F\otimes G(\omega^1,\dots, \omega^{l+q},X_1,\dots, X_{k+p}) = F(\omega^1,\dots, \omega^{l},X_1,\dots, X_{k})G(\omega^{l+1},\dots, \omega^{l+q},X_{k+1},\dots, X_{k+p}).
    \end{equation} 
\end{defn}
If \((E_1,\dots ,E_n)\) is a basis for \(V\) and \((\varphi^1,\dots \varphi^n)\) denotes the dual basis for \(V^*\), i.e. \(\varphi^i (E_j) = \delta^i_j\), then a basis for \(T^k_l(V)\) is,
\begin{equation}
    E_{j_1}\otimes \dots \otimes E_{j_l} \otimes \varphi ^{i_1} \otimes \dots \otimes \varphi ^{i_k},
\end{equation}
where the indices \(i_p,j_q\) range from \(1\) to \(n\) and we can write any tensor \(F\in T^k_l(V)\) as,
\begin{equation}
    F = F^{j_1\dots j_l}_{i_1\dots i_k} E_{j_1}\otimes \dots \otimes E_{j_l} \otimes \varphi ^{i_1} \otimes \dots \otimes \varphi ^{i_k},
\end{equation}
where
\begin{equation}
    F^{j_1\dots j_l}_{i_1\dots i_k} = F(\varphi ^{j_1}, \dots,\varphi ^{j_l},E_{i_1},\dots ,E_{i_k}).
\end{equation}
More succinctly, a \(T^k_l(M)\) tensor can be written in the form \(V^{\otimes l}\otimes (V^*)^{\otimes k}\).

If we consider the tangent space \(V=T_p M\) for some \(p\in M\), and if \((x^i)\) are local coordinates on \(U\subset M\), and \(p\in U\), then we have a basis \(\{\partial _i\}\) and dual basis \(\{dx^i\}\) of the cotangent space \(T^*_p(M)\). Then using the above formulae for \(F\in T^k_l(T_p M)\),
\begin{equation}
    F = F^{j_1\dots j_l}_{i_1\dots i_k} \partial _{j_1} \otimes \dots \otimes \partial _{j_l} \otimes dx^{i_1} \otimes dx^{i_k}.
\end{equation}

\begin{defn}[Bundle of \({k\choose l}\)-tensors on \(M\)]
    Define the bundle of \({k\choose l}\)-tensors on \(M\) as the disjoint union,
    \begin{equation}
        T^k_l M :=  \bigsqcup_{p\in M} T^k_l (T_p M).
    \end{equation}
\end{defn}

Let \(\pi :E\rightarrow M\), be a vector bundle over \(M\) where \(E\) and \(M\) are smooth manifolds and the projection \(\pi\) is a surjective map by definition. 

\begin{defn}[Smooth Section of a Vector Bundle]
    A section of a vector bundle is a map, \(F:M\rightarrow E\), such that \(\pi \circ F = Id_M\) or equivalently, \(F(p)\in E_p := \pi^{-1}(p)\). A section is smooth if it smooth as a map of manifolds, that is if and only if the components \(F^{j_1\dots j_l}_{i_1\dots i_k}\) of \(F\) depend smoothly on \(p\in U\), in any smooth local frame \(\{E_i\}\).
\end{defn}
The space of smooth sections on a vector bundle \(\pi: E\rightarrow  M\) on a smooth manifold \(M\) is denoted by \(\Gamma (E,M)\).

\begin{defn}[Tensor Field]
    A tensor field on \(M\) is a smooth section of some tensor bundle \(T^k_l M\).
\end{defn}
Hence we can say a \({k\choose l}\)-tensor field is a section \(T\in \Gamma \left(T^k_lM,M\right)\). The space of \({k\choose l}\)-tensor fields is denoted by \(\mathfrak{T}^k_l(M)\). 

\begin{defn}[Riemannian metric]
A \textit{Riemannian metric} on a smooth manifold \(M\) is a \(2\)-tensor field, \(g\in \mathfrak{T}^2(M)\) such that,

\begin{itemize}
    \item \(g\) is symmetric, \(g(X,Y)=g(Y,X)\), \(\forall X,Y\in T_pM\) and,
    \item \(g(X,X)\geq 0\) \(\forall X \in T_pM\),
    \item \(g(X,X)=0 \iff X=0\).
\end{itemize}
\end{defn}
This is an inner product on each tangent space, \(\langle X,Y \rangle := g(X,Y)\). The pair, \((M,g)\) is a \textit{Riemannian manifold} and \(g\) always exists using a partitions of unity argument.
\newline 
If \((\partial_1,\dots ,\partial_n)\) is a local frame of \(TM\), and \((dx^1,\dots, dx)\) its dual coframe of \(T^*M\) , \(g\) can be written locally as,
\begin{equation}
    g = g_{ij} dx^i \otimes dx^j, 
\end{equation}
with the coefficient matrix, \(g_{ij} = \langle \partial_i, \partial_j \rangle\). Using the symmetric product of two 1-forms, \(\omega \eta := \frac{1}{2} \left( \omega \otimes \eta + \eta \otimes \omega \right)\), this can be shortened to just,
\begin{equation}
    g = g_{ij} dx^i dx^j.
\end{equation}

Consider raising and lowering indices. Given a metric \(g\) on \(M\), define a map called \textit{flat} from \(TM\) to \(T^*M\) by mapping a vector \(X\) to the covector \(X^{\flat}\), defined by,
\begin{equation}
    X^{\flat}(Y) := g(X,Y)
\end{equation}
and in local coordinates,
\begin{equation}
    \begin{split}
        X^{\flat} & = g(X^i\partial _i, \cdot)\\
        & = g_{ij}X^i dx^j\\
        & = X_j dx^j.
    \end{split}
\end{equation}
Hence \(X^{\flat}\) is obtained from \(X\) by lowering an index and the matrix of flat is just the matrix of \(g\). Denote the inverse of flat by a map called \textit{sharp}, \(\omega \mapsto \omega ^{\#}\). In coordinates,
\begin{equation}
    \omega ^i := g^{ij}\omega _j,
\end{equation}
where \(g^{ij}\) are the components of the inverse matrix \((g_{ij})^{-1}\). Now we can define the gradient of a function \(f\) on a Riemannian manifold \((M,g)\), as the vector field grad\(f:= df^{\#}\), which is obtained from \(df\) by raising an index. The gradient is characterised from the fact \(df(Y) = \langle \text{grad}f,Y\rangle, \forall Y\in TM\). Then locally,
\begin{equation}
    \text{grad}f = g^{ij}\partial_i f \partial _j.    
\end{equation}




\subsection{Riemannian Volume Form and Integration on Riemannian Manifolds}

\begin{defn}[Frame and Coframe]
    A local frame for \(TM\) is a set of vector fields, \(\{E_i\}_{i=1}^n\), defined on an open subset \(U\subset M\), such that evaluating at \(p\in U\) gives a basis for \(T_pM\). In local coordinates \((x^1,\dots,x^n)\), we can write,
    \begin{equation}
        E_i(p) = X_i^j \frac{\partial}{\partial x^j}\bigg| _p,
    \end{equation}
    where \(X_i^j\) are matrix coefficients for these coordinates. Similarly, a local coframe for \(T^*M\) is a set of covector fields \(\{\varphi^i\}_{i=1}^n\), , defined on an open subset \(U\subset M\), such that evaluating at \(p\in U\) gives a basis for \(T^*_pM\). Locally,
    \begin{equation}
        \varphi^i(p) = Y^i_j dx^j\big|_p.
    \end{equation}
\end{defn}

\begin{lemma}[Riemannian Volume Form]\label{RVF}
    Given an oriented Riemannian \(n\)-manifold \((M,g)\), there exists a unique \(n\)-form \(dV_g\) characterised by any of the following equivalent properties:
    \newline
    a)  If \((\varphi^1,\dots,\varphi^n)\) is a local oriented orthonormal coframe for  \(T^*M\), then
    \begin{equation}
        dV_g = \varphi^1 \wedge \dots \wedge \varphi^n, 
    \end{equation}
    b)  If \((E_1,\dots,E_n)\) is a local oriented orthonormal frame for  \(TM\), then
    \begin{equation}
        dV_g(E_1,\dots,E_n) = 1,
    \end{equation}
    c) In oriented local coordinates \((x^1,\dots,x^n)\),
    \begin{equation}
        dV_g = \sqrt{\text{det}(g_{ij})} dx^1 \wedge \dots \wedge dx^n,
    \end{equation}
    where \(g_{ij} = \langle \partial_i,\partial_j \rangle\) are the coefficients of the Riemannian metric \(g\).
\end{lemma}

\begin{proof}
   Define \(dV_g\) by the equation in c). Check that this definition is well-defined, that is, it is independent of choice of local oriented coordinates. Consider \((x^i)\) and \((\tilde{x}^i)\) are local oriented coordinates of \(M\). Transform to the \((\tilde{x})^i\) frame, using Einstein's summation notation, by,
    \begin{equation}
        d\tilde{x}^a = \frac{\partial \tilde{x}^a}{\partial x^i}dx^i,
    \end{equation}
    and so the \(n-\)wedge product transforms to,
    \begin{equation}
        \begin{split}
            d\tilde{x}^1\wedge \dots \wedge d\tilde{x}^n & = \bigg(\frac{\partial \tilde{x}^1}{\partial x^{i_1}}dx^{i_1} \bigg) \wedge \dots \wedge \bigg(\frac{\partial \tilde{x}^n}{\partial x^{i_n}}dx^{i_n} \bigg)\\
            & = \text{det}\bigg(\frac{\partial \tilde{x}^a}{\partial x^i}\bigg) dx^1\wedge \dots \wedge dx^n.
        \end{split}
    \end{equation}
    Then the Riemannian metric is,
    \begin{equation}
        \begin{split}
            g & = g_{ij}dx^i\otimes dx^j\\
            & = \tilde{g}_{ab} d\tilde{x}^a\otimes d\tilde{x}^b\\
            & = \tilde{g}_{ab} \bigg(\frac{\partial \tilde{x}^a}{\partial x^i} dx^i \bigg) \otimes \bigg(\frac{\partial \tilde{x}^b}{\partial x^j} dx^j \bigg)\\
            & = \bigg( \frac{\partial \tilde{x}^a}{\partial x^i} \tilde{g}_{ab} \frac{\partial \tilde{x}^b}{\partial x^j} \bigg) dx^i \otimes dx^j,
        \end{split}
    \end{equation}
    and so \(g_{ij} = \frac{\partial \tilde{x}^a}{\partial x^i} \tilde{g}_{ab} \frac{\partial \tilde{x}^b}{\partial x^j}\). Taking determinants,
    \begin{equation}
        \text{det}(g_{ij}) = \text{det}(\tilde{g}_{ab}) \bigg(\text{det}\bigg(\frac{\partial \tilde{x}^c}{\partial x^k}\bigg) \bigg)^2,
    \end{equation}
    and so we get,
    \begin{equation}
        \begin{split}
            \sqrt{\text{det}(\tilde{g}_{ab})} d\tilde{x}^1 \wedge \dots \wedge d\tilde{x}^n & = \frac{\sqrt{\text{det}(g_{ij})}}{\big|\text{det}\big(\frac{\partial \tilde{x}^c}{\partial x^k}\big) \big|} \text{det}\big(\frac{\partial \tilde{x}^c}{\partial x^k}\big) dx^1 \wedge \dots \wedge dx^n\\
            & = \sqrt{\text{det}(g_{ij})} dx^1 \wedge \dots \wedge dx^n,
        \end{split}
    \end{equation}
    as \((x^i)\) and \((\tilde{x}^j)\) are locally oriented coordinates.
    Now to prove the equivalence of the 3 statements. 
    \newline
    The matrix representation of \(g\) wrt \((E_1,\dots ,E_n)\) is the identity matrix, so we have b). The matrix representation of \(g\) wrt \((\varphi^1,\dots ,\varphi^n)\) is \((g_{ij})\), as defined before. Let \(A_p\) be the transition matrix between these two frames of \(T_pM\). We have, by b), with local coordinates \(x^1,\dots,x^n\),
    \begin{equation}
        dV_g(\partial_1,\dots,\partial_n)_p = \text{det}A_p \cdot dV_g(E_1,\dots,E_n) = \text{det}A_p,
    \end{equation}
    or
    \begin{equation}
        (dV_g)_p = \text{det}A_p dx^1\wedge \dots \wedge dx^n.
    \end{equation}
    
   By construction,
   \begin{equation}
       (g_{ij})_p = A_p^T \cdot I \cdot A_p = A_p^T A_p,
   \end{equation}
   so
   \begin{equation}
       \text{det}A_p = \sqrt{\text{det}(A^T_pA_p)} = \sqrt{\text{det}(g_{ij})}
   \end{equation}
    and we get,
    \begin{equation}
        dV_g = \sqrt{\text{det}(g_{ij})} dx^1 \wedge \dots \wedge dx^n,
    \end{equation}
    which is c). We can reverse all the arguments, so we have the equivalence of the 3 statements.
\end{proof}
The \(n\)-form \(dV_g\) is called the \textit{Riemannian volume form}. Now we can integrate functions on \(M\), not just differential forms. If \(f\) is a smooth, compactly supported function on an oriented Riemannian \(n-\)fold manifold \((M,g)\), then \(fdV_g\) is a compactly supported \(n-\)form, so the integral \(\int_M fdV_g\) is well-defined. The volume of \(M\) is defined to be \(\int_M dV_g = \int_M 1 dV_g\). Note that if \(M\) is compact, then all functions on it are compactly supported.

Let \(\Sigma^k \subset M^n\) be an embedded submanifold and \(\iota :\Sigma^k\xhookrightarrow{} (M^n,g)\) the inclusion map, which is an embedding. This induces a Riemannian metric on \(\Sigma\), via a pullback: \(g_{\Sigma}= \iota^* g\). Assuming \(\Sigma\) and \(M\) are both oriented, \(g_{\Sigma}\) has the volume form,
\begin{equation}
    dV_{g_{\Sigma}} = dV_{\Sigma} = dV_{\iota^*g}.
\end{equation}
This is called the \(k-\)dimensional Hausdorff measure, \(d\mathcal{H}^k\), and allows us to integrate over (sub)-manifolds over smaller dimension than the dimension of \(M\), as follows.
\begin{equation}
    \int_{\Sigma} d\mathcal{H}^k = \int_{\Sigma}dV_{\iota^*g }
\end{equation}
It may be useful to denote \(dV_{\iota^*g }\) as \(dA_{\iota^*g }\), to note it is a volume of a lower-dimensional space.

Next we need a few definitions so that we can prove when a hypersurface in an oriented Riemannian manifold is itself orientable, which boils down to the existence of a well-defined global continuous unit normal vector field on the hypersurface. For intuition, you can think of a 2-dimensional plane embedded in \(\mathbb{R}^3\). The plane is fully determined by the span of 2 orthonormal tangent vectors lying on the plane and the area form is the projection of the volume form of \(\mathbb{R}^3\) onto the plane, where \(\mathbb{R}^3\) is fully determined by the 2 orthonormal tangent vectors and a third normal vector to the plane. 

\begin{defn}[Smooth Hypersurface]
    A hypersurface is a manifold of dimension \(n-1\), which is embedded in an ambient space of dimension \(n\). A smooth hypersurface is a hypersurface that is a smooth manifold. 
\end{defn}
Particularly we can consider a hypersurface \(\Sigma ^{n-1}\), as above, embedded in \(M^n\), as a codimension 1 submanifold.

Define the \textit{interior multiplication} of an alternating \(k-\)tensor, \(\omega \in \bigwedge^k(V^*)\) by vector \(v\in V\), by the map, \(\omega \rightarrow v \lrcorner \omega\), where \(\lrcorner\) is defined by,
\begin{equation}
    (v\lrcorner \omega)(w_1,\dots,w_{k-1}) = \omega (v,w_1,\dots,w_{k-1}).
\end{equation}
Note that this holds for a vector field and differential forms.
\begin{defn}[Slice Charts]
    Suppose \(n>k\) and write \(\mathbb{R}^n = \mathbb{R}^k\times \mathbb{R}^{n-k}\). Given an open subset \(U_p\subset M\) and for a fixed \(p\in U_p\), there exists a chart of \(M\),
    \begin{equation}
        (\tilde{U},\tilde{\phi}) \text{ such that } (\tilde{\phi})^{-1}(\mathbb{R}^k\times \{0\}) = \tilde{U}\cap U_p,
    \end{equation}
    then \(U_p\) is a \(k-\)dimensional submanifold of \(M\) and \((\tilde{U},\tilde{\phi})\) is called a slice chart. Particularly, if \(M\) has local coordinates, \((x^1,\dots,x^n)\),  where \(x^{k+1},\dots ,x^n\) vanish on \(U_p\), and the slice chart  \((\tilde{U},\tilde{\phi})\) of  \(M\) has local coordinates \((x^1,\dots,x^k,t^1,\dots,t^{n-k})\), where each \(t^i\) are constants in \(\mathbb{R}\). Note that if all \(t^i=0\), this is just \(U_p\subset \mathbb{R}^k\) canonically embedded in \(\mathbb{R}^n\).
\end{defn}

\begin{defn}[Global Continuous Normal Vector Field]
        Let \((M,g)\) be a Riemannian manifold and \(S\) a submanifold.            Then the map defined by, 
        \begin{equation}
            N:S\rightarrow TM, p\mapsto N(p)\in T_p M,
        \end{equation}
        that satisfies,
        \begin{equation}
            N(p)\in (T_pS)^{\perp}\setminus \{0\}, \hspace{1mm}\forall p \in M,
        \end{equation}
        is called a normal vector field, where the orthogonal complement is taken with respect to \(g\). \(N\) is continuous at \(p\in S\) if for a chart \((\mathcal{U},\phi)\) of \(M\), where \(p\in \mathcal{U}\), \(N\) can be written in the form,
        \begin{equation}
            N(x) = \sum_i a_i(x) \partial^{\phi}_i(x),
        \end{equation}
        where \(a_i:\mathcal{U}\rightarrow \mathbb{R}\) are continuous functions and \(\{\partial_i^{\phi}\}\) forms a basis of the tangent space, \(T_xM\). \(N\) is a continuous unit normal vector field if additionally it satisfies,
        \begin{equation}
            ||N(x)|| = \sqrt{g(N(x),N(x))} = 1, \forall x\in S.
        \end{equation}
        \(N\) is a global normal vector field if it is well-defined on the whole manifold \(M\).
\end{defn}

Now for the orientability condition. 

\begin{prop}[Orientability Condition]
    Suppose \(M\) is a hypersurface in an oriented Riemannian manifold \((\tilde{M},\tilde{g})\) and \(g\) is the induced metric on \(M\). Then \(M\) is orientable if and only if there exists a global continuous unit normal vector field \(N\) for \(M\), and in that case, the volume form of \((M,g)\) is,
    \begin{equation}
        dV_g = (N\lrcorner dV_{\tilde{g}})|_M.
    \end{equation}
\end{prop}
\begin{proof}
    (\(\implies\)) Assume \(M\) is orientable and for \(p\in M\) let \(\mathcal{U}_p\) with coordinates, \((x^1,\dots,x^{n-1},t)\), be a slice chart of \(\tilde{M}\), where \(t=0\) corresponds to \(M\) (as an embedded submanifold of \(\tilde{M}\)).  Then as the differential preserves orientation, we can pick on \(\mathcal{U}_p\) either \(N_{\mathcal{U}}=\partial _t\) or \(N_{\mathcal{U}}=-\partial _t\), for a given connected component, depending on whichever makes \(\{\partial_1,\dots,\partial_{n-1},\partial_t\}\) positively oriented in the connected component of \(\tilde{M}\), which is determined by the sign of the orientation chosen on \(\tilde{M}\). For every point \(p\in M\), it exists in some connected component of \(M\), and we can make this choice of orientation, so \(N_M\) is globally defined, where the union of connected components cover \(M\). We can choose a partition of unity \(\{\theta_i\}_{i\in \mathcal{I}}\) subordinate to the cover of \(M\), \(\{\mathcal{U}_p\}_{p\in M}\), then if we write the cover as \(\{\mathcal{U}_{p_i}\}_{i\in \mathcal{I}}\) of \(M\), for \(p_i \in M\), with a countable index set \(\mathcal{I}\) and connected open subsets \(\mathcal{U}_{p_i}\), then,
    \begin{equation}
        N = \sum_i \theta_i N_{\mathcal{U}_{p_i}},
    \end{equation}
    is a global vector field on \(M\), with \(N_{\mathcal{U}_{p_i}}\) as the choice of normal vector field on the connected component \(\{\mathcal{U}_{p_i}\}\). Note that this makes \(N\) non-zero, as \(\forall i \in \mathcal{I}, \exists p_i\) such that, \(\text{supp}(\theta_i) \subset \mathcal{U}_{p_i}\), from the definition of partitions of unity.
    \newline
    \(N\) is clearly continuous and by the definition of coordinates, \(\pm\partial_t\) is not in the span of all the \(\partial_i\), so \(\pm \partial_t \notin TM\). Hence we must have \(\pm \partial_t\) non-zero in \(T\tilde{M}|_M \setminus TM\).  \(N\) is not necessarily normal to \(TM\) with respect to the metric \(g\), i.e. \(N(p)\notin (T_pM)^{\perp}\setminus \{0\}, \forall p\in M\). 
    
    To find a normal vector field we must subtract off the parallel part of \(N\) and check this is non-zero and normal to the \((\partial _i)\)'s. 
    Define,
    \begin{equation}
        N_{\parallel} := \sum_{i,j} \langle N,\partial_i\rangle \partial_i g^{ij} \in \text{Span}(\partial_i) \text{ and } \tilde{N}:= N - N_{\parallel}.
    \end{equation}
    Check \(N_{\parallel}\) is well-defined. Change to coordinates \((\tilde{x}^i)\), with,
    \begin{equation}
        \tilde{\partial_i} = \frac{\partial x^j}{\partial\tilde{x}^i}\frac{\partial}{\partial x^j} \hspace{3mm} \text{ and } \tilde{g}^{ij} = g^{\alpha \beta} \frac{\partial x_{\alpha}}{\partial\tilde{x}^i} \frac{\partial x_{\beta}}{\partial\tilde{x}^j},
    \end{equation}
    to get,
    \begin{equation}
        \begin{split}
            \sum_{i,j} \langle N,\tilde{\partial_i}\rangle \tilde{\partial_i} \tilde{g}^{ij} & = \sum_{i,j} \langle N,\frac{\partial x^j}{\partial\tilde{x}^i} \frac{\partial}{\partial x^j}\rangle \tilde{\partial_i} g^{\alpha \beta} \frac{\partial x_{\alpha}}{\partial\tilde{x}^i} \frac{\partial x_{\beta}}{\partial\tilde{x}^j}\\
            & = \sum_{i,j} \langle N, \frac{\partial}{\partial x^j}\rangle \tilde{\partial_i} g^{\alpha \beta} \frac{\partial x^j}{\partial\tilde{x}^i}\frac{\partial x_{\alpha}}{\partial\tilde{x}^i} \frac{\partial x_{\beta}}{\partial\tilde{x}^j}\\ 
            & = \sum_{i,j} \langle N, \frac{\partial}{\partial x^j}\rangle \tilde{\partial_i} g^{\alpha \beta} \delta_{\alpha i}\delta_{\beta j}\\
            & = \sum_{i,j} \langle N, \partial_i\rangle \partial_i g^{ij}\\
            & = N_{\parallel}.
        \end{split}
    \end{equation}

    Then for a fixed \(i\in\{1,\dots,n-1\}\), assuming \((\partial_i)\) is an orthonormal basis,
    \begin{equation}
        \begin{split}
            \langle \tilde{N}, \partial_i\rangle & = \langle N, \partial_i\rangle - \bigg{\langle} \sum_{j,k} \langle N,\partial_j\rangle\partial_j g^{jk}, \partial_i\bigg{\rangle}\\
            & = \langle N, \partial_i\rangle - \sum_{j,k} \langle N,\partial_j\rangle\langle \partial_j, \partial_i\rangle g^{jk}\\
            & = \langle N, \partial_i\rangle - \sum_{j,k} \langle N,\partial_j\rangle \delta_{ij} g^{jk}\\
            & = \langle N, \partial_i\rangle - \sum_{k} \langle N,\partial_i\rangle g^{ik} \\
            & = \langle N, \partial_i\rangle - \sum_{k} \langle N,\partial_i\rangle \delta_{ik}\\
            & = (1-g^{ii})\langle \partial_t, \partial_i\rangle = 0,
        \end{split}
    \end{equation}
    since we have scaled so \(g^{ii} = \langle \partial_i,\partial_i\rangle = 1\) and so \(N\) is normal to every \(\partial_i\). If we don't have an orthonormal basis, we need to have at least a unit length basis of tangent vectors, so for each \(i=1,\dots ,n-1\), \(\langle \partial_i,\partial_i\rangle = 1\).
    \newline
    Finally to check the unit property, we need to instead consider the vector field \(V(x) := \frac{\tilde{N}(x)}{||\tilde{N}(x)||}\). Check that \(\tilde{N}(x)\neq 0 \hspace{2mm} \forall x \in M\). We have,
    \begin{equation}
        \begin{split}
            \tilde{N}= 0 & \iff N = N_{\parallel}\\
            & \iff N = \sum_{i,j} \langle N, \partial_i\rangle \partial_i g^{ij}\\
            & \iff N = \pm \partial_t \text{ is linearly dependent on } \{\partial_1,\dots,\partial_{n-1}\},
        \end{split}
    \end{equation}
    then \(\partial_t\) would not be a basis vector of \(T\tilde{M}\).
    
    In coordinates, \(x=(x^1,\dots,x^{n-1},t)\),
    \begin{equation}
        dx^i\otimes dx^j(V(x),V(x)) = dx^i\bigg(\frac{\partial_t}{||\partial_t||}\bigg)\otimes dx^j\bigg(\frac{\partial_t}{||\partial_t||}\bigg) = 1, 
    \end{equation}
    where the subscript \(t\) corresponds to the fixed \(n\)-th coordinate. In this basis \(g_{ij}= \langle\frac{\partial_i}{||\partial_i||},\frac{\partial_j}{||\partial_j||}\rangle\), so
    \begin{equation}
        \begin{split}
            ||V(x)||^2 & = g(V(x),V(x))\\
            & = g_{ij}dx^i\otimes dx^j(V(x),V(x))\\
            & =\frac{\partial_i\cdot \partial_j}{||\partial_i||\cdot ||\partial_j||}\\
            & = \frac{||\partial_i||\cdot ||\partial_j|| \cos{\frac{\pi}{2}}}{||\partial_i||\cdot ||\partial_j||}\\
            &= 1,
        \end{split}
    \end{equation}
    where we have assumed \((\partial_i)\) is an orthonormal basis.
    Then the volume form on \(\tilde{M}\) is, noting \(\text{det}(g_{ij})\neq 0\) by definition,
    \begin{equation}
       dV_{\tilde{g}} = \sqrt{\text{det}(\tilde{g}_{ij})} dx^1\wedge \dots \wedge dx^{n-1}\wedge dt = \sqrt{\frac{\text{det}(\tilde{g}_{ij})}{\text{det}(g_{ij})}} dV_g|_M \wedge dt.
    \end{equation}
    Use \(V = V(x^1,\dots,x^{n-1},t)\),
    \begin{equation}
        \begin{split}
            (V\lrcorner dV_{\tilde{g}})(x^1,\dots,x^{n-1})|_M 
            & = dV_{\tilde{g}}(V,x^1,\dots,x^{n-1})|_M\\
            & = \sqrt{\frac{\text{det}(\tilde{g}_{ij})}{\text{det}(g_{ij})}} dV_g (x^1,\dots,x^{n-1})\wedge dt(V)|_M \\
            & = \sqrt{\frac{\text{det}(\tilde{g}_{ij})}{\text{det}(g_{ij})}} dV_g (x^1,\dots,x^{n-1}),
        \end{split}
    \end{equation}
using that \(dt(V) = dt(\frac{\partial_t}{||\partial_t||}) = \frac{1}{||\partial_t||}\). Using the chosen coordinates, notice that,

\[ \tilde{g}_{ij} =
\begin{pmatrix}
  \biggij
  & \rvline & \begin{matrix}
  \partial_1 \cdot \partial_t \\
  \dots \\
  \partial_{n-1} \cdot \partial_t
  \end{matrix} \\
\hline
  \begin{matrix}
  \partial_t \cdot \partial_1 & \dots & \partial_t \cdot \partial_{n-1}
  \end{matrix} & \rvline &
  \begin{matrix}
  \partial_t \cdot \partial_t
  \end{matrix} \\
\end{pmatrix}.
\] Denote \(\tilde{G}_{i,j}\) to be the minor of \(\tilde{g}_{ij}\), by removing row \(i\) and column \(j\) and taking the determinant. Expanding along the final column,

\begin{equation}
    \text{det}(\tilde{g}_{ij}) = \partial_1 \cdot \partial_t \tilde{G}_{1,t} + \partial_2 \cdot \partial_t \tilde{G}_{2,t} + \dots + \partial_{n-1} \cdot \partial_t \tilde{G}_{n-1,t} + \partial_t \cdot \partial_t \tilde{G}_{t,t}.
\end{equation}
This means every \(\tilde{G}_{i,j}=0\) (see \autoref{sec:app_B}) apart from \(\tilde{G}_{t,t}=\text{det}(g_{ij})\), which immediately gives \(\text{det}(\tilde{g}_{ij}) = \partial_t \cdot \partial_t \text{det}(g_{ij}) = ||\partial _t||^2 \text{det}(g_{ij})\), and,
\begin{equation}
    \sqrt{\frac{\text{det}(\tilde{g}_{ij})}{\text{det}(g_{ij})}} = ||\partial _t||,
\end{equation}
cancelling with the \(\frac{1}{||\partial_t||}\), to give the result,
\begin{equation}
    (V\lrcorner dV_{\tilde{g}})(x^1,\dots,x^{n-1})|_M = dV_g (x^1,\dots,x^{n-1}).
\end{equation}
This can also be realised by noting,
\begin{equation}
    N\lrcorner dV_{\tilde{g}} (E_1,\dots,E_{n-1}) = 1,
\end{equation}
whenever \((E_1,\dots,E_{n-1})\) is an oriented orthonormal frame of \(TM\), so \((N,E_1,\dots,E_{n-1})\) is an oriented orthonormal frame of \(T\tilde{M}\), which can be shown to by extending to an open neighbourhood of the point. Hence we get by the equivalence of conditions in \ref{RVF},

\begin{equation}
    N\lrcorner dV_{\tilde{g}} (E_1,\dots, E_{n-1}) = dV_{\tilde{g}}(N,E_1,\dots,E_{n-1}) = 1
\end{equation}
and \(N\lrcorner dV_{\tilde{g}} = dV_g\).

(\(\impliedby\)) Given a global unit normal vector field, \(N\), use the same slice chart as before around \(p\in M\). Now given \(N\), orient \(M\cap \mathcal{U}_p\) by declaring the ordered basis \(\{\partial_1,\dots,\partial_{n-1}\}\) to be positively oriented on \(M\cap \mathcal{U}_p\) if and only if the ordered basis \(\{\partial_1,\dots,\partial_{n-1},\partial_t\}\) is positively oriented on \(\tilde{M}\cap \tilde{\mathcal{U}}_{p}\), where \(\tilde{\mathcal{U}}_{p}\) has coordinates \((x^1,\dots,x^{n-1},t)\) when \(\mathcal{U}_p\) has coordinates \((x^1,\dots ,x^{n-1})\).

Now check this works on any intersections, \(\mathcal{U}_p\cap \mathcal{U}_q\), with \(p,q\in M\). 

Since \(\tilde{M}\) is orientable, we have, in local oriented coordinates, \((x^1,\dots,x^{n-1},t)\), the Riemannian volume form,
\begin{equation}
    dV_{\tilde{g}} = \sqrt{\text{det}(\tilde{g}_{ij})}dx^1\wedge \dots \wedge dx^{n-1}\wedge dt = (N\lrcorner dV_{\tilde{g}})|_M,
\end{equation}
by assumption. Now, again since \(\tilde{M}\) is orientable, there exists a cover \(\{\tilde{\mathcal{U}}_{\alpha}\}\) of \(\tilde{M}\), such that \(\text{det}\bigg(\frac{\partial y^{\beta}}{\partial x^{\alpha}} \bigg) >0\).  Hence for any \({\mathcal{U}}_{\alpha}, {\mathcal{U}}_{\beta}\) , such that \({\mathcal{U}}_{\alpha}\cap {\mathcal{U}}_{\beta} \neq \emptyset\), which implies \(\tilde{\mathcal{U}}_{\alpha}\cap \tilde{\mathcal{U}}_{\beta} \neq \emptyset\),
\begin{equation}
    \begin{split}
        dV_{\tilde{g}}|_{\mathcal{U}_{\alpha}} & = (N\lrcorner dV_{\tilde{g}})|_{\tilde{\mathcal{U}}_{\alpha}}\\
        & = \text{det}\bigg(\frac{\partial y^{\beta }}{\partial x^{\alpha }} \bigg) (N\lrcorner dV_{\tilde{g}})|_{\tilde{\mathcal{U}}_{\beta}}\\
        & = \text{det}\bigg(\frac{\partial y^{\beta}}{\partial x^{\alpha }} \bigg)dV_{\tilde{g}}|_{\mathcal{U}_{\beta}}.
    \end{split}
\end{equation}
    Then, the LHS is \(>0 \iff\) the RHS is \(>0\), so it holds over intersections.
\end{proof}

\begin{prop}[The Riemannian Density]
    If \((M,g)\) is any Riemannian manifold, then there is a unique smooth positive density \(\mu\) on \(M\), called the Riemannian density, with the property that,
    \begin{equation}
        \mu (E_1,\dots,E_n)=1,         
    \end{equation}

    for every local orthonormal frame \((E_i)\). Particularly, for the coframe \((\varphi^1,\dots,\varphi^n)\), \(\mu = |\varphi^1 \wedge \dots \wedge \varphi^n|\).
\end{prop}

\subsection{Symmetric Decreasing Rearrangement Inequalities on \(M^n\)}

Now we have a well-defined definition of integrating \(f:M\rightarrow \mathbb{R}\), when \(f\) is a smooth and compactly supported function on our oriented Riemannian manifold \((M^n=(0,\infty)\times \Sigma^{n-1},g)\). We can consider the \(L^p\)-norm of \(f\), 
\begin{equation}
    ||f||_p = \bigg(\int _M |f|^p dV_g\bigg)^{1/p}.
\end{equation}
Then the \(L^p(\Omega)\) space for \(\Omega \subset M\), is defined as the space of functions \(f:\Omega \rightarrow \mathbb{R}\), such that \(||f||_p < \infty\).
Let \(f\) be a smooth, compactly supported function on an oriented Riemannian \(n-\)fold manifold \((M,g)\), then using the definition of volume on a manifold, define the distribution function as before,
\begin{equation}
    \mu_f(t) = \text{Vol}_{g}\left(\{x:f(x)>t\} \right) = \int_{\{f>t\}} \,dV_g.
\end{equation}

The formula for the area of a ball of radius \(x\in \mathbb{R}^n\) is,
\begin{equation}
    f(|x|) = \omega_n |x|^n = \int_{(0,|x|)\times S^{n-1}} r^{n-1} drdA.
\end{equation}
We can use this idea in the manifold setting, by considering the integral,
\begin{equation}
    \int_{(0,r)\times \Sigma} \omega,
\end{equation}
where \(r>0\), \(S^{n-1}\leftrightarrow \Sigma^{n-1}\) and \(\omega \leftrightarrow r^{n-1} drdA\). We have already proven the following theorem for the case of subsets of \(\mathbb{R}^n\), so if \(M\) can be embedded in some \(\mathbb{R}^N\), then the following already holds. Otherwise, it can be proven similarly with the new manifold definitions. 

\begin{thm}[Rearrangement of level sets is equivalent to level sets of rearrangement]
    The two sets are equal, \(\{x\in M: f^*(x) > t\} = \{x\in M: f(x) >t\}^*\).
\end{thm}
\begin{proof}
    The aim is to mimic the proof from the \(\mathbb{R}^n\) case and adapt it using the manifold setting definitions. Fix a \(t>0\) and \(y\in \{x:f(x)>t\}^* =\{f>t\}^*\subset M\). Write the layer-cake decomposition of \(f^*\) and split as before,
    \begin{equation}
        \begin{split}
            f^*(y) & = \int_0^{\infty} \mathcal{X} _{\{x:f(x)>s\}^*} (y) ds\\
            & = \int_0^t \mathcal{X} _{\{x:f(x)>s\}^*} (y) ds + \int_t^{\infty} \mathcal{X} _{\{x:f(x)>s\}^*} (y) ds.
        \end{split}
    \end{equation}
    The first integral is again just \(t\). Suppose for a contradiction that \(\int_t^{\infty} \mathcal{X} _{\{x:f(x)>s\}^*} (y) ds = 0\), so \(\forall s > t,  y\notin \{f>s\}^*\). Then using the formula above, taking \(y=(r_y,\sigma _y)\in (0,\infty )\times \Sigma = M\) and \(\omega\) as the volume form we have,
    \begin{equation}
        |\{f>s\}| = \int_{\{f>s\}} \omega.
    \end{equation}
    For all \(s>t\) we have \(y\notin \{f>s\}^* = (0,r_*(s))\times \Sigma\) which implies, \(r_y \geq r_*(s), \forall s>t\). Then using the rearrangement of subsets of \(M\),
    \begin{equation}
        \begin{split}
            \int_{(0,r_y)\times \Sigma} \omega & \geq \int_{(0,r_*(s))\times \Sigma} \omega \\
            & = \int_{\{f>s\}^*} \omega\\
            & = \int_{\{f>s\}} \omega\\
            & \geq \limsup_{s\rightarrow t} \int_{\{f>s\}} \omega\\
            & = \int_{\{f>t\}} \omega,
        \end{split}
    \end{equation}
    which is a contradiction and implies \(\{f>t\}^* \subset \{f^*>t\}\) as before. The reverse inclusion follows the same as the \(\mathbb{R}^n\) case.
\end{proof}

With this theorem proven, the following corollary follow exactly as before in the \(\mathbb{R}^n\) case.

\begin{corollary}
    For any \(t>0, \mu _f(t) = \mu _{f^*}(t)\).
\end{corollary}

\begin{lemma}[Symmetric Decreasing Rearrangement Preserves \(L^p\)-norms]
    For all non-negative \(f\in L^p(M)\), \(1\leq p \leq \infty\), with local oriented coordinates \((x^i)\) of \(M\),
    \begin{equation}
        ||f(x)||_p = ||f^*(x)||_p, \hspace{3mm} \forall x\in  M,
    \end{equation}
    where \(x = (x^1,\dots,x^n)\).
\end{lemma}

\begin{corollary}[Symmetric Decreasing Rearrangement is Order-Preserving]
\end{corollary}
\begin{equation}
    f(x)\leq g(x) \hspace{1mm} \forall x\in M \implies f^*(x)\leq g^*(x) \hspace{1mm} \forall x\in M.
\end{equation}

\begin{thm}[Hardy-Littlewood]
    Let \(g\) be of the same form as \(f\), with local oriented coordinates \((x^i)\) of \(M\), then,
    \begin{equation}
        \int_M f(x)g(x) \,dV_g \leq \int_M f^*(x)g^*(x) \,dV_g,
    \end{equation}
    with \(x = (x^1,\dots,x^n)\).
\end{thm}

\begin{thm}[Rearrangement Decreases \(L^p\) Distance]
    Given \(f,g\) as before and \(1\leq p \leq \infty\),
    \begin{equation}
        ||f(x)-g(x)||_p \geq ||f^*(x)-g^*(x)||_p \hspace{2mm} \forall x\in M.
    \end{equation}
\end{thm}

We can see that plenty of theorems generalise nicely to the Riemannian manifolds setting.

\section{Smooth Co-area Formula for Riemannian Manifolds Proof}\label{sec:co_area}

In the \(\mathbb{R}^n\) case, we needed the coarea formula for splitting an \(n\)-dimensional integral into an \((n-1)\)-dimensional integral over level sets and a 1-dimensional integral over all of these level sets. We would like to reproduce this formula for Riemannian manifolds, but in more generality. Particularly, if we have an \(m\)-dimensional Riemannian manifold, then we would like to write this as an integral over an \((m-n)\)-submanifold and an \(n\)-dimensional submanifold. The \((m-n)\)-integral picks up the idea of level sets as before, where the \(n\)-integral acts as the parameter integral. Again we use the Hausdorff measure, now defined on Riemannian manifolds. 

The proof consists of many smaller lemmas, following Chavel's textbook \cite{RM_chavel}.

We can generalise the previous discussion of the Jacobian to smooth manifold maps. Given Riemannian manifolds \((M^m,g_M), (N^n,g_N)\), with \(m\geq n\), and a \(C^1\) map, \(\Phi:M\rightarrow N\), the Jacobian of \(\Phi\) is defined as,
\begin{equation}
    \mathcal{J}_{\Phi}(x) = \sqrt{\text{det}\hspace{1mm}\Phi_* \circ (\Phi _*)^{adj}},
\end{equation}
where \(\Phi_*:TM\rightarrow TN\), \((p,v_p)\mapsto (\Phi(p),v_{\Phi(p)})\), is the push-forward map of tangent vectors and \((\Phi_*)^{\text{adj}}:TN\rightarrow TM\) is the adjoint map. Then the composite map is then \(\Phi_*\circ (\Phi_*)^{\text{adj}}:TN\rightarrow TN\). Write \(\Phi_{*|x}:T_xM\rightarrow T_{\Phi(x)}N\) for a specific point \(x\in M\) and then considering the natural vector space isomorphism, \(T_{\Phi(p)}N\cong \mathbb{R}^{n}\), let \(\Phi_*\circ (\Phi_*)^{\text{adj}}\) have an \(n\times n\) matrix representation \(A\). Then the \(n\)-dimensional volume of the n-parallelepiped is given by,,
\begin{equation}
    \begin{split}
        \mathcal{J}_A & = \sqrt{\text{det}(A^TA)}\\
        & = \text{det}\big(\langle v_i,v_j \rangle _{T_pM} \big)\\
        & = \langle v_1\wedge \dots \wedge v_n, v_1\wedge \dots \wedge v_n\rangle _{\Lambda^n T_pM},
    \end{split}
\end{equation}
where the tangent space \(T_pM\cong \mathbb{R}^{m}\) as a vector space isomorphism also. This is also called the Gram determinant, \(\mathbb{G}(v_1,\dots,v_n)\).

Let \(\{e_i\}, \{\omega ^i\}\) be an orthonormal moving frame (respectively dual coframe) on \(M\) and similarly \(\{E_j\}, \{\varphi ^j\}\) on \(N\). Note that we can pullback 1-forms on \(N\) to \(M\), that is there exist functions \(\sigma ^j_i\) on \(M\), for \(i=1,\dots,m\), such that, \(\Phi ^* \varphi ^j = \sigma ^j_i \omega ^i\), where \(\Phi ^*:TN \rightarrow TM\) is the pullback of tangent vectors from \(T_pM\) to \(T_{\Phi(p)}N\).

\begin{lemma}[Form of Jacobian]\label{form_of_Jac}
    \begin{equation}
        \mathcal{J}_{\Phi}(x) = \left\{
        \begin{array}{ll}
              0 & , \text{rank}\hspace{1mm}\Phi _* < n \\
              |\text{det}(\Phi_*|(\text{ker}\Phi_*)^{\perp})| & , \text{rank}\hspace{1mm}\Phi _* = n \\
        \end{array} 
        \right.
    \end{equation},
    where \((\text{ker}\Phi_*)^{\perp}=\text{span}\{a_1^T,\dots, a_n^T\}\) , where \(\Phi_*\) has rows, \(a_i\).
\end{lemma}

\begin{proof}
    If we identify each \(T_pM \cong T_p\mathbb{R}^m\) and \(T_{\Phi(p)}N \cong T_{\Phi(p)}\mathbb{R}^n\), which are isomorphic as vector spaces, then we can apply the rank-nullity theorem. If \(\text{rank}\Phi_* < n\), then \(\text{rank}(\Phi_*\circ (\Phi_*)^{\text{adj}})  < n\) , and by rank-nullity theorem,  \(\text{dim ker}(\Phi_*\circ (\Phi_*)^{\text{adj}}) > n-n =0\). Hence the \(n\times n\) matrix representation of \(\Phi_*\circ (\Phi_*)^{\text{adj}}\) has at least one zero rows, so \(\text{det}\hspace{1mm}\Phi_*\circ (\Phi_*)^{\text{adj}}\textbf{} = 0\) and \(\mathcal{J}_{\Phi}(x) = 0\) in this case.
Similarly, if \(\text{rank}\Phi_*|_x =n\), for some \(x\in M\),  then \(\text{dim ker}(\Phi_*|_x) = m-n\), so \(x\) is a regular point of \(\Phi\). Then, \(\Phi\) is a submersion on a neighbourhood of \(x\), so given a local fibration about \(x\), we can pick \(e_{n+1},\dots,e_m\) tangent to the fibres. That means \(e_1,\dots,e_n\) are orthogonal to this fibration and,
\begin{equation}
    \sigma^j_l = 0, \hspace{4mm} j=1,\dots,n,  \hspace{4mm} l=n+1,\dots,m.
\end{equation}
Then,
\begin{equation}
    \Phi ^* \varphi ^j = \sigma ^j_k \omega ^k, \hspace{2mm} j,k=1,\dots,n,
\end{equation}
and,
\begin{equation}    
    \begin{split}
        \Phi^* (\varphi^1\wedge \dots \wedge \varphi^n) & = \text{det}(\sigma^j_k) \omega^1 \wedge \dots \wedge \omega ^n\\
        & = \mathcal{J}_{\Phi}\omega^1 \wedge \dots \wedge \omega ^n,
    \end{split}
\end{equation}
where,
\begin{equation}
    \text{det}(\sigma^j_k) = \text{det}
    \begin{pmatrix}
    \sigma^1_1 & \dots & \sigma^n_1\\
    \vdots & \ddots & \vdots\\
    \sigma^1_n & \dots & \sigma^n_n\\
    \end{pmatrix} = \text{det}\big(\Phi_* | (\text{ker}\Phi_*)^{\perp}\big).
\end{equation}

Concluding two cases gives the result, noting \(\mathcal{J}_{\Phi} \geq 0\). 

\end{proof}

The argument with the local fibration naturally comes from thinking about a vector space as, \(V=\text{ker}(T)\oplus T(V)\), where \(T:V\rightarrow V\) is a linear map.

We need to consider how \(\Phi\) pulls-back \(n\)-forms on \(N\) to \(m\)-forms on \(M\).

\begin{lemma}
    If for some \(x\in M\), we have that \(\text{dim}(\text{ker}\Phi_{*|x})>m-n\), then,
    \begin{equation}
        \Phi^* (\varphi^1\wedge \dots \wedge \varphi^n) = 0,
    \end{equation}
    at \(x\).
\end{lemma}
\begin{proof}
    Follows directly from the following lemma, as  \(\text{dim}(\text{ker}\Phi_{*|x})>m-n \implies \text{rank}\Phi_{*|x} < n\), so \(\mathcal{J}_{\Phi|x}=0\) and \(\omega^{n+1} \wedge \dots \wedge \omega^m \neq 0\). Equivalently, \(x\) is a critical point of \(\Phi\).
\end{proof}

\begin{lemma}[Equivalence of Forms on \(M\)]
    If for some \(x\in M\), we have that dim \((\)ker\(\Phi_{*|x})=m-n\), then,
    \begin{equation}
        \mathcal{J}_{\Phi}\omega^1 \wedge \dots \wedge \omega ^m = \Phi^* (\varphi^1\wedge \dots \wedge \varphi^n) \wedge \omega^{n+1}\wedge \dots \wedge \omega^{m},
    \end{equation}
    at \(x\).
\end{lemma}
\begin{proof}
    If we have the nice basis from Lemma \ref{form_of_Jac}, then we have 
    \begin{equation}
    \Phi^* (\varphi^1\wedge \dots \wedge \varphi^n) = \mathcal{J}_{\Phi} \omega^1 \wedge \dots \wedge \omega ^n,
\end{equation}
so we can wedge \(\omega^{n+1}\wedge \dots \wedge \omega^m\) onto each side to get,
\begin{equation}
    \mathcal{J}_{\Phi} \omega^1 \wedge \dots \wedge \omega ^m = \Phi^* (\varphi^1\wedge \dots \wedge \varphi^n) \wedge\omega^{n+1}\wedge \dots \wedge \omega^m. 
\end{equation}
\end{proof}
Now suppose \(\Phi:M\rightarrow N\) is \(C^r\), for some \(r>m-n\), and \(M,N\) as before. Let \(f:M\rightarrow \mathbb{R}\) be a measurable, non-negative function in \(L^1(M)\). We need to employ Sard's theorem to prove the following theorem.





\begin{thm}[Sard's Theorem]
    Let \(\Phi:M^m \rightarrow N^n\) be a smooth map of manifolds.  Let \(X(\Phi)\) denote the set of critical points of \(\Phi\); that is the set of all \(x\in M\) with,
    \begin{equation}
        \text{rank}\hspace{1mm} d\Phi|_x < n.
    \end{equation}
    Then, \(X(\Phi)\subset N\) has measure zero in \(N\).
\end{thm}
\begin{proof}
    Refer to Milnor's book "Topology from the differentiable viewpoint", chapter 3.
\end{proof}

Now we can quickly deduce the main result of this section.

\begin{thm}[Generalised Co-area Formula]
    Define \(\Phi:M\rightarrow N \in C^r, \hspace{1mm} r > m-n\), and \(f:M\rightarrow \mathbb{R}\in L^1(M)\), then,
    \begin{equation}
        \int_M f \mathcal{J}_{\Phi} d\mathcal{H}^{m} = \int_N d\mathcal{H}^{n}(y) \int_{\Phi^{-1}[y]} \left(f|_{\Phi^{-1}[y]} \right) d\mathcal{H}^{m-n},
    \end{equation}
    where \(d\mathcal{H}^{n} = dV_{\iota^*g}\) as before.
\end{thm}
\begin{proof}
    Apply the previous lemma, with \(dV_m = \omega^1 \wedge \dots \wedge \omega^m\), and denoting the set of regular values of \(\Phi\) as \(R(\Phi)\).
    \begin{equation}
        \begin{split}
            \int_M f \mathcal{J}_{\Phi} d\mathcal{H}^{m} & = \int_M f \Phi^* (\varphi^1\wedge \dots \wedge \varphi^n) \wedge \omega^{n+1}\wedge \dots \wedge \omega^{m},\\
            & = \int_{R(\Phi)}d\mathcal{H}^{n}(y) \int_{\Phi^{-1}[y]} \left(f|_{\Phi^{-1}[y]} \right) d\mathcal{H}^{m-n},
        \end{split}
    \end{equation}
    where \(y\in R(\Phi)\), so the level sets \(\Phi^{-1}[y]\) are well-defined (Riemannian) submanifolds of \(M\). Since \(\Phi\in C^r\), \(R(\Phi)\) has at most measure zero in \(N\), by Sard's theorem.
\end{proof}
Now let \(\Phi: (M,g)\rightarrow \mathbb{R} \in C^m\). Given \(f\) as before, we can obtain the previous co-area formula, like in the \(\mathbb{R}^n\) case.
\begin{thm}[Co-area Formula]
\end{thm}
\begin{equation}
   \int_M f |\text{grad}\Phi| dV_g = \int_{\mathbb{R}} dy \int_{\Phi^{-1}[y]} \left(f|_{\Phi^{-1}[y]} \right) d\mathcal{H}^{m-1}.        
\end{equation}

Note that \(\Phi^{-1}[y] = \{x\in M: \Phi(x)=y\}=\{\Phi=y\}\), which is the boundary of the set \(\{x\in M: \Phi(x)>y\} = \{\Phi>y\}\).

\begin{proof}
    Apply the previous lemma with \(n=1\), so \(\Phi\) can be associated with the map, \(\Phi_{\mathbb{R}}: \mathbb{R}^m \rightarrow \mathbb{R}\), with \(T_pM \cong \mathbb{R}^m, T_{\Phi(p)}N \cong T_{\Phi(p)}\mathbb{R}\), as before. As \(n=1\), \(\text{det}\hspace{1mm}\Phi_* \circ (\Phi _*)^{adj}\) is just the single value of  the \(1\times 1\) matrix \(\Phi_* \circ (\Phi _*)^{adj} = (\text{grad}\Phi)^2\), so \(\mathcal{J}_{\Phi} = |\text{grad}\Phi|\). Using the volume form on a Riemannian manifold, \(dV_g = \omega^1 \wedge \dots \wedge \omega^n\), and denoting \(d\mathcal{H}^{1}(y)=dy\) gives the result.
\end{proof}

\subsection{Perimeter}

The co-area formula now allows us to generalise the idea of perimeter to \(M^n\).

\begin{defn}[Perimeter]
    The perimeter of and open subset \(A\subset M\) is,
\end{defn}
\begin{equation}
    \text{Per}(A) := \int_{\partial A} d\mathcal{H}^{m-1}
\end{equation}

Choosing \(f\equiv 1\) in the co-area formula gives,
\begin{equation}
    ||\text{grad}\Phi||_1 =  \int_{\mathbb{R}} dy \int_{\{\Phi=y\}}  d\mathcal{H}^{m-1} = \int_{\mathbb{R}} \text{Per}(\{\Phi>y\})dy.
\end{equation}
Now we have the tools to recast the isoperimetric and P\'{o}lya-Szeg\H{o} inequalities in the Riemannian manifold setting. 

\begin{defn}[Isoperimetric Inequality]
    For any open subset \(A\subset M^n:= (0,\infty)\times \Sigma^{n-1}\), it satisfies the isoperimetric inequality if,
    \begin{equation}
        \text{Per}(A) := \int_{\partial A} d\mathcal{H}^{n-1} \geq \int_{\partial A^*} d\mathcal{H}^{n-1} =: \text{Per}(A^*),
    \end{equation}
    where \(A^* = (0,r_*)\times \Sigma\), as before. 
\end{defn}

\begin{defn}[P\'{o}lya-Szeg\H{o} Inequality]
    Given a measurable smooth function, \(f:M\rightarrow \mathbb{R}_+\), it satisfies the P\'{o}lya-Szeg\H{o} inequality if,
    \begin{equation}
        ||\text{grad} f||_p \geq ||\text{grad} f^*||_p,
    \end{equation}
    where \(\text{grad}f := g^{ij}\partial_i f \partial _j\), as before.
\end{defn}

We can apply the same approximation in proof (1) in section \ref{approx_proofs} to show that the P\'{o}lya-Szeg\H{o} inequality implies the isoperimetric inequality.

\begin{appendix}

\section{Existence and Uniqueness of the Symmetric Rearrangement of \(A^*\subset M\)}\label{sec:app_A}
\setcounter{lemma}{0}
\renewcommand{\thelemma}{\Alph{section}\arabic{lemma}}

\begin{lemma}
    \(A^*\) exists and is unique.
\end{lemma}

\begin{proof}
    \textbf{Existence}: Fix \(A \subset M\). Consider the continuous function \(f:(0, \infty) \to \mathbb R\) given by
    \begin{equation}
        f(r) = \int_{(0,r) \times \Sigma} \omega.
    \end{equation}
    Then, if \(A\subset M\) is finite and has nonzero measure, \(\exists \hspace{1mm} r_1 ,r_2 \in (0,\infty)\), such that \(r_1 < r_2\) and,
    \begin{equation}
        f(r_1) = \int_{(0,r_1)\times \Sigma} \omega < \int _A \omega < \int_{(0,r_2)\times \Sigma} \omega = f(r_2).
    \end{equation}
    If \(A\) is all of \(M\), take \(r_* = \infty\), and if \(A=\emptyset\), take \(r_*=0\). Now by the intermediate value theorem, \(\exists \hspace{1mm} r_*\in (0,\infty )\), such that,
    \begin{equation}
        \int _A \omega = f(r_*) = \int_{(0,r_*)\times \Sigma} \omega = \int _{A^*} \omega.
    \end{equation}

    \textbf{Uniqueness}: Suppose for \(A\subset M\) there exist two symmetric rearrangements, \(A^* = (0,r_*)\times \Sigma\) and \(\tilde{A}^* = (0,\tilde{r}_*)\times \Sigma\) such that \(r_* \neq \tilde{r}_*\), 
    \begin{equation}
        \int_ {A^*} \omega = \int_ A \omega \text{ and } \int_ {\tilde{A}^*} \omega = \int_ A \omega.
    \end{equation}
    Suppose, \(r_* > \tilde{r}_*\) WLOG. Hence, \(\int_ {A^*} \omega = \int_ {\tilde{A}^*} \omega\), and by definition 
    \begin{equation}
        \int_ {(0,r_*)\times \Sigma} \omega = \int_ {(0,\tilde{r}_*)\times \Sigma} \omega, 
    \end{equation}
    
    \begin{equation}
        0 = \int_ {(0,r_*)\times \Sigma} \omega - \int_ {(0,\tilde{r}_*)\times \Sigma} \omega\ = \int_ {(\tilde{r}_*,r_*)\times \Sigma} \omega,
    \end{equation}
    but \(\omega > 0 \), so \((\tilde{r}_*,r_*)\) has measure \(0\) and we must have \(r_* = \tilde{r}_*\) and we have uniqueness.
\end{proof}
A potentially cleaner way to consider these integrals is to write \(\omega\) in polar coordinates as, \(\omega = dr\wedge \alpha_r\), detailed in the following lemma.

\begin{lemma}
    Given an \(n\)-form \(\omega\) on \(M^n\) as before, we can uniquely write \(\omega\) as the wedge of a non-vanishing \((n-1)\)-form \(\alpha _r\), on \(\Sigma ^{n-1}\), and the standard polar coordinate, \(dr\). That is, \(\omega = dr \wedge \alpha _r\) is independent of coordinate choice.
\end{lemma}

\begin{proof}
    Suppose \(M\) has local coordinates, \(x_1, \dots ,x_n\), we can write,
    \begin{equation}
        \begin{split}
            \omega \left( x^1, \dots ,x^n \right) & = \varphi (x^1,x^2, \dots , x^n)dr\wedge dx^2 \wedge \dots \wedge dx^n,\\
            & = dr \wedge \left (\varphi (r,x^1, \dots , x^n) dx^2 \wedge \dots \wedge dx^n\right ) = dr\wedge \alpha _r,
        \end{split}
    \end{equation}
        taking \(x^1 = r\) and \(\varphi :M^n \rightarrow \mathbb{R}, \varphi >0\) is a smooth function, where we have assumed \(\omega >0\) for the positively oriented basis \(x^1, \dots ,x^n\). Transform to another positively oriented coordinate frame \(y^1 = r, y^2, \dots ,y^n\), and write \(Y=y^1,\dots y^n\),
    \begin{equation}
        \begin{split}
            \omega \left( r,x^1(Y),\dots ,x^n(Y) \right) & = \varphi (r,x^1(Y),\dots ,x^n(Y)) \hspace{1mm} \text{det} \left( \frac{\partial x^i}{\partial y^j} \right)_{1\leq i,j, \leq n} dr \wedge \dots \wedge dy^n\\
            & = \psi (r,y^1,\dots ,y^n) dr\wedge \dots \wedge dy^n\\
            & = dr \wedge \alpha _r,
        \end{split}
    \end{equation}
    where \(\psi :M^n \rightarrow \mathbb{R}, \psi >0\) is a smooth function, since det\(\left( \frac{\partial x^i}{\partial y^j} \right) >0\), as \(M^n\) is orientable.
\end{proof}

In the uniqueness proof of \(A^*\), we could equivalently write \(\int _{A^*}\omega = \int_{(0,r_*)} dr \int_{\Sigma}\alpha _r\) and using Fubini's theorem to conclude similarly. The following line would instead read,
\begin{equation}
    0 = \int_{(0,r_*)} dr \int_{\Sigma}\alpha _r - \int_{(0,\tilde{r_*})} dr \int_{\Sigma}\alpha _r = \int_{(\tilde{r_*},r_*)} dr \int_{\Sigma}\alpha _r,
\end{equation}
where \(\int_{\Sigma}\alpha _r > 0\), so vanishes \(\int_{(\tilde{r_*},r_*)} dr\) and \((\tilde{r_*},r_*)\) is a set of measure zero.

\section{Determinant Calculation for Orientability Condition}\label{sec:app_B}

\end{appendix}

\begin{equation}
    \begin{split}
        \tilde{G}_{ij} & = \text{det}
    \begin{pmatrix}
      \begin{matrix}
      \partial_1 \cdot \partial_1 & \dots &\widehat{\partial_1 \cdot \partial_j} & \dots & \partial_1 \cdot \partial_t \\
      \vdots & & \vdots & & \vdots \\
      \widehat{\partial_{i} \cdot \partial_1} & \dots & \widehat{\partial_{i} \cdot \partial_j} & \dots & \widehat{\partial_{i} \cdot \partial_t}\\
      \vdots & & \vdots & & \vdots \\
      \partial_{t} \cdot \partial_1 & \dots & \widehat{\partial_{t} \cdot \partial_j} & \dots & \partial_{t} \cdot \partial_t
      \end{matrix}
    \end{pmatrix}\\ & = \text{det}
    \begin{pmatrix}
        \begin{matrix}
            \partial_1 \cdot \partial_1 & \dots & \partial_{1} \cdot \partial_{j-1} & \partial_{1} \cdot \partial_{j+1} & \dots & \partial_{1} \cdot \partial_{t}\\
            \vdots & \ddots & & & & \vdots\\
            \partial_{i-1} \cdot \partial_{1} & & \ddots & & & \partial_{i-1} \cdot \partial_{t}\\
            \partial_{i+1} \cdot \partial_{1} & & & \ddots & & \partial_{i+1} \cdot \partial_{t}\\
            \vdots & & & & \ddots & \vdots\\
            \partial_t \cdot \partial_1 & \dots & \partial_{t} \cdot \partial_{j-1} & \partial_{t} \cdot \partial_{j+1} & \dots & \partial_{t} \cdot \partial_{t}\\
        \end{matrix}
    \end{pmatrix}\\
    & = \sum_{\sigma}\text{sgn}(\sigma) \partial_1 \cdot \partial_{\sigma(1)} \dots \widehat{\partial_i \cdot \partial_{\sigma(i)}} \dots \widehat{\partial_j \cdot \partial_{\sigma(j)}} \dots \partial_t \cdot \partial_{\sigma(t)},
    \end{split}
\end{equation}

which is non-zero only when \(i,j = t\), since all \(\partial_i \cdot \partial_t = 0\).

\end{document}